\documentclass[10pt]{amsart}

\usepackage{amscd,amsfonts,amssymb}
\usepackage{tikz}
\usepackage{pgfplots}
\pgfplotsset{compat=1.18}
\usetikzlibrary{automata,positioning,calc,trees}
\usetikzlibrary{positioning,arrows.meta}
\usetikzlibrary{intersections,pgfplots.fillbetween}
\usepackage{graphicx,tikz}
\usepackage{graphicx} 
\usepackage{pgf,tikz,pgfplots}
\usepackage{mathrsfs}
\usetikzlibrary{arrows}
\usepackage{enumerate}
\usepackage{subfig}
\usepackage[shortlabels]{enumitem}
\usepackage{mathrsfs}
\usepackage{amssymb,amsmath,amsthm,color}
\usepackage{caption,subcaption}
\usepackage[alphabetic,bibtex-style]{amsrefs}
\usepackage{hyperref}
\usepackage{url}
\usepackage{setspace}
\usepackage{float}
\tikzstyle{startstop} = [rectangle, rounded corners, minimum width=3cm, minimum height=1cm,text centered, draw=black]
\tikzstyle{process} = [rectangle, rounded corners, minimum width=3cm, minimum height=1cm, text centered, draw=black]
\tikzstyle{arrow} = [thick,->,>=stealth]
\BibSpec{article}{%
	+{}  {\PrintAuthors}                {author}
	+{,} { \textit}                     {title}
	+{.} { }                            {part}
	+{:} { \textit}                     {subtitle}
	+{,} { \PrintContributions}         {contribution}
	+{.} { \PrintPartials}              {partial}
	+{,} { }                            {journal}
	+{}  { \textbf}                     {volume}
	+{}  { \PrintDatePV}                {date}
	+{,} { \issuetext}                  {number}
	+{,} { \eprintpages}                {pages}
	+{,} { }                            {status}
	+{,} { \url}                        {url}    
	+{,} { \PrintDOI}                   {doi}
	+{,} { available at \eprint}        {eprint}
	+{}  { \parenthesize}               {language}
	+{}  { \PrintTranslation}           {translation}
	+{;} { \PrintReprint}               {reprint}
	+{.} { }                            {note}
	+{.} {}                             {transition}
}

\date{\today}
\allowdisplaybreaks
\newcommand{\R}{{\mathbb R}}       
\newcommand{\N}{{\mathbb N}}

\newcommand{\Z}{{\mathbb Z}}       

\newtheorem{theorem}{Theorem}[section]
\newtheorem{lemma}[theorem]{Lemma}
\newtheorem{remark}[theorem]{Remark}

\newtheorem{coro}[theorem]{Corollary}
\newtheorem{proposition}[theorem]{Proposition}

\newtheorem*{theorem*}{Theorem}

\theoremstyle{definition}

\usepackage{color}

\numberwithin{equation}{section}

\pgfplotsset{compat=1.18}
\begin{document}

	\title{The bilinear cone multiplier on $\R^2\times \R^2$}

	\begin{abstract}
    In this paper, we study the bilinear cone multiplier operator in two dimensions. We establish $L^{p_1}\times L^{p_2}\to L^{p}$ boundedness for a regularized version of this operator over a broad range of exponents satisfying the H\"older scaling condition. Our approach is based on a decomposition of the bilinear operator into square functions associated with linear cone multipliers and their variants. We derive pointwise bounds for these square functions via suitable strong maximal function estimates, and obtain sharp $L^4$ bounds using geometric methods originating in the work of C\'ordoba and Carbery. The combination of these estimates yields the $L^p$ boundedness for the bilinear cone multiplier.
	\end{abstract}

\author[L. Roncal]{Luz Roncal}
\address[Luz Roncal]{
	BCAM -- Basque Center for Applied Mathematics,
	48009 Bilbao, Spain,
	\newline \phantom{\quad} \&
	Ikerbasque, Basque Foundation for Science,
	48011 Bilbao, Spain,
	\newline \phantom{\quad} \&
	Universidad del Pa\'is Vasco / Euskal Herriko Unibertsitatea,
	48080 Bilbao, Spain}
\email{\href{mailto:lroncal@bcamath.org}{lroncal@bcamath.org}}

\author[S. Shrivastava]{Saurabh Shrivastava} 	
\address[Saurabh Shrivastava]{Indian Institute of Science Education and Research Bhopal, 462066, India}
\email{saurabhk@iiserb.ac.in}

\author[K. Shuin]{Kalachand Shuin}
\address[Kalachand Shuin]{Department of Mathematics, Indian Institute of Science, Bangalore-560012, India}
\email{shuin.k.c@gmail.com, kalachands@iisc.ac.in}
\author[L. Zheng]{Linfei Zheng}
\address[Linfei Zheng]{
    Center for Applied Mathematics, Tianjin University, Weijin Road 92, 300072 Tianjin, China		
    \newline \phantom{\quad} \&
    BCAM -- Basque Center for Applied Mathematics, 48009 Bilbao, Spain}
\email{\href{mailto:linfei_zheng@tju.edu.cn}{linfei\_zheng@tju.edu.cn}}

	\maketitle
	
\section{Introduction}

The study of the linear cone multiplier  operator goes back to Stein \cite{SteinSomeproblemsinharmonicanalysis}. Given the multiplier
$$
\xi=(\xi',\xi_n) \mapsto  \Big(1 - \frac{|\xi'|^2}{\xi_n^2}\Big)_+^\lambda, \quad (\xi',\xi_n)\in \R^{n-1}\times \R, \quad n\ge 3,
$$
the associated operator $C^{\lambda}$ exhibits a singularity along the light cone, and its $L^p$ boundedness properties remain largely open.  
In dimension $n=2$, the singularity of this operator appears along a triangle, thus it is not that interesting due to the lack of curvature. 

It was conjectured in \cite[p.5]{SteinSomeproblemsinharmonicanalysis}
that, for $\lambda>0$ and $1\leq p\leq \infty$, the cone multiplier operator is bounded in $L^p(\R^n)$ if, and only if, 
$$
\lambda>\lambda(p)=\max \Big\{(n-1)\Big|\frac{1}{p}-\frac{1}{2}\Big|-\frac{1}{2},0\Big\}.
$$
Observe that for fixed $\lambda>0$, the conjectured range of $L^p$ boundedness of  $C^{\lambda}$ in $n$ dimensions is the same as that of the conjectured range of Bochner--Riesz means of order $\lambda$ in $(n-1)$ dimensions $B^{\lambda}$, defined by 
$$
\xi'\mapsto (1-|\xi'|^2)^{\lambda}_{+},  \quad \xi'\in \mathbb{R}^{n-1}, \quad n\geq 3.
$$
Indeed, the transference principle for Fourier multipliers asserts that the cone multiplier conjecture is stronger than the Bochner--Riesz conjecture.  

A major step toward a tractable formulation of the cone multiplier problem was introduced by Mockenhaupt \cite{Mockenhaupt}, who inserted a smooth cutoff in the transversal variable. More precisely, one considers multipliers of the form
\begin{equation}
\label{eq:regularized}
(\xi',\xi_n) \mapsto \Big(1 - \frac{|\xi'|^2}{\xi_n^2}\Big)_+^\lambda \varphi(\xi_n),
\end{equation}
where $\varphi \in C_c^\infty([1/2,2])$. 
This modification preserves the essential geometry of the cone while eliminating certain degeneracies, and provides a useful model for the analysis of cone-type singularities and associated square functions. In the linear setting, such operators are closely connected to conical square functions, Kakeya-type phenomena, and directional maximal operators, and their study has led to major developments in harmonic analysis.

We briefly recall the historical context. Mockenhaupt \cite{Mockenhaupt} reduced the study of the regularized cone multiplier \eqref{eq:regularized} to suitable $L^p$ estimates for an associated square function and established an $L^4$ bound for $\lambda>\frac{1}{8}$. Bourgain~\cite{BourgainEstimatesforCone} later obtained improved $L^p$ estimates using bilinear methods. Over the past two decades, this direction has been extensively developed; see, for example, \cites{GustavoSeegerCone, TaoVargasBilinearApproachConeII, WolffLocalSmoothing, Heo, HHY, LeeVargas, GuthWangZhangAsharpSquareFunctionEstimate}. In particular, Lee and Vargas~\cite{LeeVargas} proved sharp $L^3$ estimates for the cone multiplier in $\mathbb{R}^3$, and the conjecture in this setting (dimension $n=3$) was eventually resolved by Guth, Wang, and Zhang~\cite{GuthWangZhangAsharpSquareFunctionEstimate}. We also refer to the recent work of Chen, He, Li, and Yan~\cite{CHLY} on pointwise almost everywhere convergence.

Bilinear analogues of singular multipliers have been also broadly studied in the context of Bochner--Riesz operators, defined by
\[
\mathcal{B}^{\lambda}(f,g)(x)
=
\int_{\mathbb{R}^{2n}}
\big(1-|\xi|^{2}-|\eta|^{2}\big)_{+}^{\lambda}
\widehat{f}(\xi)\widehat{g}(\eta)
e^{2\pi i x\cdot(\xi+\eta)}
\, d\xi \, d\eta,
\quad
f,g\in\mathcal S(\R^n).
\]
When $\lambda=0$, the multiplier reduces to the characteristic function of the unit ball in $\R^{2n}$, which is singular along the sphere. While in the linear setting such operators are bounded only on $L^2$ in dimension $n\ge2$~\cite{Fefferman}, the bilinear theory exhibits a richer range of boundedness phenomena. Bilinear Bochner--Riesz operators have been studied extensively in recent years. We refer to~\cite{GrafakosLiTheDiscBilinearMultiplier,BGSY,LiuWang,JLV,JL,JS} and the references therein for a detailed account of known results and developments.

The purpose of this paper is to study a bilinear analogue of the regularized cone multiplier in dimension two. For $\lambda>0$, consider the bilinear operator in dimension $n$, initially defined for $f,g\in \mathcal S(\R^n)$,
\begin{equation}
\label{eq:bilinear}
\mathcal{C}^\lambda(f,g)(x) = \int_{\mathbb{R}^{2n}} \Big(1 - \frac{|\xi'|^2}{\xi_n^2} - \frac{|\eta'|^2}{\eta_n^2}\Big)_+^\lambda  \varphi(\xi_n)\varphi(\eta_n)\widehat f(\xi)\widehat g(\eta) e^{2\pi i x\cdot(\xi+\eta)}\, d\xi d\eta,
\end{equation}
where $\xi=(\xi',\xi_n)$ and $\eta=(\eta',\eta_n)\in \R^{n-1}\times \R$. This operator can be viewed as a bilinear counterpart of Mockenhaupt's cone multiplier.

A useful feature of this formulation is that, by the transference principle for bilinear multipliers (see, for instance, \cites{Oscar,Shobha,Grafakosclassical}), $L^{p_1}(\R^n)\times L^{p_2}(\R^n)\to L^{p}(\R^n)$ boundedness of $\mathcal C^{\lambda}$ implies the corresponding $L^{p_1}(\R^{n-1})\times L^{p_2}(\R^{n-1})\to L^{p}(\R^{n-1})$ bounds for bilinear Bochner--Riesz means of the same index $\lambda>0$ when $n\ge2$. In this sense, estimates for $\mathcal C^\lambda$ may be viewed as a strengthening of those for bilinear Bochner--Riesz means. This perspective also reflects the additional analytical difficulties inherent in the problem: unlike the elliptic geometry underlying Bochner--Riesz multipliers, the cone exhibits anisotropic scaling and pronounced directional interactions. These features introduce subtleties absent in the elliptic setting and necessitate different techniques in the analysis.

 In this article we establish $L^{p_1}(\R^2)\times L^{p_2}(\R^2)\to L^{p}(\R^2)$ boundedness results for the bilinear cone multiplier \eqref{eq:bilinear} in dimension $n=2$, going beyond the natural $(2,2,1)$ endpoint. This is the first non-trivial setting in which such $L^p$ estimates are investigated for this operator. In dimension $n=2$, the operator $\mathcal C^{\lambda}$, initially defined for $f,g\in \mathcal S(\R^2)$, reads explicitly as
\[
\mathcal C^{\lambda}(f,g)(x):=\int_{\mathbb{R}^{4}}\Big(1-\frac{\xi_1^{2}}{\xi_2^2}-\frac{\eta_1 ^{2}}{\eta_2^2}\Big)^{\lambda}_{+}\widehat{f}(\xi)\widehat{g}(\eta)\varphi (\xi_2)\varphi(\eta_2) e^{2\pi ix\cdot(\xi+\eta)}\,d\xi d\eta.
\]

The study of bilinear cone multipliers was initiated by the second and third authors~\cite{shrivastavashuin_cone}, who obtained almost everywhere convergence results in dimensions $n\geq 3$. Their approach, building on ideas from~\cites{CHLY,CarberyRubioVega,JS}, relies on weighted $L^{2}\times L^{2}\to L^{1}$ estimates for maximal operators associated with scaled bilinear cone multipliers\footnote{The restriction to dimensions $n\ge 3$ in \cite{shrivastavashuin_cone} stems from the use of an estimate for the linear cone multiplier established in \cite{CHLY}. Since the result in \cite{CHLY} is formulated for $n\ge 3$, the argument in \cite{shrivastavashuin_cone} applies only in this range without further modifications. It is plausible that the corresponding statements remain valid in dimension $n=2$, although this case is not addressed there.}. In contrast, the present work focuses directly on $L^p$ boundedness for the operator $\mathcal C^\lambda$ itself, rather than its maximal variant.

The restriction to dimension $n=2$ is not merely technical. In this case, the multiplier exhibits a genuinely two-dimensional geometric structure, allowing for the use of classical geometric methods -such as those related to directional decompositions and almost orthogonality- that are less effective or presently unavailable in higher dimensions. 
To the best of our knowledge, $L^{p_1}\times L^{p_2}\to L^p$ boundedness for $\mathcal C^\lambda$ has not been previously established even in dimension two, and the corresponding problem in higher dimensions remains open. Our results provide a first step in this direction by treating a broad range of exponents in the planar case.

\subsection{Main result and method of proof}

The main result of this paper is the following. Set
\[
p_-:=\min\{p_1,p_2\}, 
\qquad 
p_+:=\max\{p_1,p_2\}.
\]

\begin{theorem}\label{main theorem}
Suppose $1 < p_1,p_2 \le \infty$ and $1 \le p \le \infty$ satisfy $\frac1p=\frac1{p_1}+\frac1{p_2}$.
Then the bilinear operator $\mathcal C^\lambda$ satisfies
\begin{equation}
\label{eq:mainestimate}
\|\mathcal C^{\lambda}(f,g)\|_{L^p(\R^2)}
\le C
\|f\|_{L^{p_1}(\R^2)}
\|g\|_{L^{p_2}(\R^2)},
\end{equation}
for all $f\in L^{p_1}(\R^2)$ and $g\in L^{p_2}(\R^2)$, 
in the following cases:
\begin{enumerate}
\item[(i)] If $2 \le p_1,p_2 \le 4$ and $\lambda>0$.
\smallskip

\item[(ii)] If $2 \le p_- \le 4 \le p_+ \le \infty$ and $\lambda>\frac12-\frac{2}{p_+}$.
\smallskip

\item[(iii)] If $4 < p_1,p_2 \le \infty$ and $\lambda>1-\frac{2}{p}$.
\smallskip

\item[(iv)] If $1 \le p_- \le 2$ and
\[
\lambda>
\begin{cases}
\displaystyle \frac{2}{p_-}-1, 
& \text{if } \frac1{p_-}+\frac{2}{p_+}\ge 1, \\[8pt]
\displaystyle \frac1{p_-}-\frac{2}{p_+}, 
& \text{if } \frac1{p_-}+\frac{2}{p_+}\le 1.
\end{cases}
\]
{\label{main theorem p<2 1}}
\end{enumerate}
\end{theorem}
The admissible region for the boundedness of $\mathcal{C}_{\lambda}$ in Theorem \ref{main theorem} is depicted in Figure \ref{fig:picture}.
\begin{remark}
\rm It is natural to conjecture that \eqref{eq:mainestimate} holds for all $\lambda>0$ throughout the full range of exponents $(p_1,p_2,p)$ satisfying the H\"older scaling condition. This would be consistent with the corresponding result for the bilinear Bochner--Riesz and the bilinear ball multiplier in one dimension due to Bernicot, Grafakos, Song, and Yan~\cite[Theorem 4.1]{BGSY}.
\end{remark}

\begin{center}
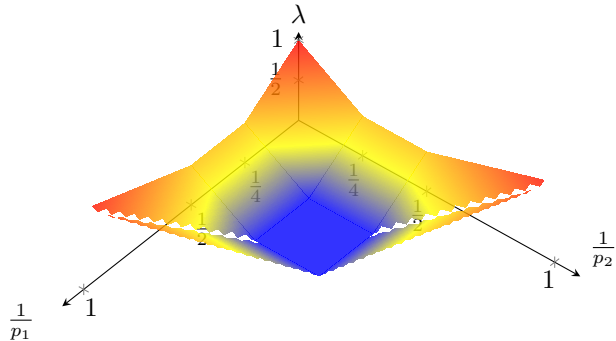
\begin{figure}[h]
\begin{tikzpicture}
\begin{axis}[
    view={130}{70},           
    axis lines=center,        
    xlabel=$\frac{1}{p_1}$, ylabel=$\frac{1}{p_2}$, zlabel=$\lambda$,
    xlabel style={at={(axis cs:1.4,0,0.1)}, anchor=south west},
    ylabel style={at={(axis cs:0,1.1,0)}, anchor=south west},
    zlabel style={at={(axis cs:0,0,1.1)}, anchor=south},
    xmin=0, xmax=1.1,
    ymin=0, ymax=1.1,
    zmin=0, zmax=1.1,
    xtick={0.25,0.5,1},
    xticklabels={\(\frac{1}{4}\),\(\frac{1}{2}\),1},
    ytick={0.25,0.5,1},
    yticklabels={\(\frac{1}{4}\),\(\frac{1}{2}\),1},
    ztick={0.5,1},
    zticklabels={\(\frac{1}{2}\),1},
    grid=major,              
    colormap/hot,           
    samples=50             
]

\addplot3[
    surf, 
    domain=0.25:0.5,        
    domain y=0.25:0.5, 
    draw=none,
    shader=interp,
    point meta=z, 
    restrict z to domain=0:1,
    opacity=0.8
] {0};

\addplot3[
    surf,                    
    domain=0.25:0.5,        
    domain y=0:0.25,  
    draw=none,
    shader=interp,
    point meta=z, 
    restrict z to domain=0:1,   
   opacity=0.8 
] {0.5-2*y};

\addplot3[
    surf,                    
    domain=0:0.25,        
    domain y=0.25:0.5,  
    draw=none,
    shader=interp,
    point meta=z, 
    restrict z to domain=0:1,   
   opacity=0.8 
] {0.5-2*x};

\addplot3[
    surf,                    
    domain=0:0.25,        
    domain y=0:0.25,  
    draw=none,
    shader=interp,
    point meta=z, 
    restrict z to domain=0:1,   
   opacity=0.8 
] {1-2*x-2*y};

\addplot3[
    surf,                    
    domain =0.5:1,        
    domain y=0:1, 
    draw=none,
    shader=interp,
    point meta=z, 
    opacity=0.8,
    restrict expr to domain={y - (0.5 -0.5*x)}{-inf:0}
 ] {x-2*y};

\addplot3[
    surf,                    
    domain =0.5:1,        
    domain y=0:1, 
    draw=none,
    shader=interp,
    point meta=z, 
    opacity=0.8,
    restrict expr to domain={y - (0.5 -0.5*x)}{0:inf},
    restrict expr to domain={y + x -1}{-inf:0},
]{2*x-1};

\addplot3[
    surf,                    
    domain =0:1,        
    domain y=0.5:1, 
    draw=none,
    shader=interp,
    point meta=z, 
    opacity=0.8,
    restrict expr to domain={x - (0.5 -0.5*y)}{-inf:0}
 ] {y-2*x};

 \addplot3[
    surf,                    
    domain =0:1,        
    domain y=0.5:1, 
    draw=none,
    shader=interp,
    point meta=z, 
    opacity=0.8,
    restrict expr to domain={x - (0.5 -0.5*y)}{0:inf},
    restrict expr to domain={x + y -1}{-inf:0},
]{2*y-1};
\end{axis}
\end{tikzpicture}
\caption{\label{figure:exponent-region}
Admissible region for $(1/p_1,1/p_2,\lambda)$ in Theorem~\ref{main theorem}. The piecewise surface represents the threshold conditions on $\lambda$ arising from the different cases of the theorem; it is conjectured that the entire region above $\lambda=0$ is admissible.
}
\label{fig:picture}
\end{figure}
\end{center}

We briefly outline the strategy of the proof. We begin with the dyadic decomposition of the bilinear cone multiplier introduced in~\cite{shrivastavashuin_cone} and a representation of the dyadic pieces exploited in \cite{JS} in the context of the bilinear Bochner--Riesz operator, reducing the problem to establishing suitable $L^p$ estimates for square functions associated with linear cone multipliers and their shifted variants. The main new ingredients are twofold. First, we perform a refined analysis of these square functions to obtain pointwise control via suitable strong maximal functions. Second, we establish $L^4$ estimates by adapting geometric methods inspired by C\'ordoba~\cite{Cor1982}. More precisely, we decompose the frequency space into thin trapezoidal regions and study the associated directional square functions and Kakeya-type maximal operators. While there are similarities with recent work on rough frequency decompositions and directional square functions (see~\cite{ADHPR23}), we do not rely on these developments, and any connection appears to be indirect at best. Our approach is instead more direct and tailored to the present setting. Finally, interpolation between the $L^p$ bounds obtained as a consequence of the pointwise estimates and $L^4$ estimates leads to the desired $L^{p_1}(\R^2)\times L^{p_2}(\R^2)\to L^{p}(\R^2)$ boundedness of $\mathcal C^{\lambda}$.

\subsection*{Organization of the paper}

The main steps in the proof of Theorem~\ref{main theorem} are summarized in Figure~\ref{figure1}. The argument proceeds via a dyadic decomposition of the operator, reducing to estimates for the pieces $\mathcal C_j^\lambda$.  In the case of small $\lambda$ (Proposition~\ref{T_j estimate}) we separately treat the low-frequency term ($j=1$) and the high-frequency terms ($j\ge2$). The latter are further reduced to square function estimates, which are obtained using a combination of almost orthogonality arguments and kernel bounds.
  
\begin{figure}[H]
	\centering
	\resizebox{1\textwidth}{!}{
		\begin{tikzpicture}[
			node distance = 1.4cm and 2cm,
			every node/.style = {
				draw,
				rounded corners,
				align=center,
				text width=4cm,   
				minimum height=1cm,
				font=\small         
			},
			arrow/.style = {->, thick}
			]
			
			\node (start)
			{$L^p$-estimates for $\mathcal C^{\lambda}$ \\
				Theorem~\ref{main theorem}};
			
			\node (deco) [below=of start]
			{Decomposition \\$\mathcal C^{\lambda}=\sum_{j\geq 1} \mathcal C_j^{\lambda}$ \\
				Section~\ref{Decomposition}};
			
			\node (estimateCj) [below=of deco]
			{$L^p$ bounds for $\mathcal C_j^{\lambda}$ 
				Propositions \ref{T_j estimate}, \ref{T_j estimate 2}};
			
			
			\node (strongmax) [right=of estimateCj]
			{Proof of Proposition~\ref{T_j estimate 2}\\ via kernel estimates\\
				Section~\ref{sec:dominationbystrongmax}};

			\node (jis1) [below right=of estimateCj]
			{Proof of Proposition~\ref{T_j estimate}, $j=1$\\ ad hoc decomposition\\ Section~\ref{j=1case}};
			
			\node (viasquare) [below left=of estimateCj]
			{Proof of Proposition~\ref{T_j estimate}, $j\geq2$\\ assuming square function estimates\\
				Propositions~\ref{square function Gnu}, \ref{square function Gmu}
                \\Section \ref{sec:squarefunctiondomination}};
			
			\node (l4) [below=of viasquare, xshift=6cm]
			{$L^4$ estimates of square functions \\ Section~\ref{sec:L4sf}};
			
			\node (pnot4) [below=of viasquare, xshift=-6cm]
			{$L^p$ estimates of square functions for $p\neq4$\\ Section~\ref{sec:pnot4}};
			
			\draw [arrow] (start) -- (deco);
			\draw [arrow] (deco) -- (estimateCj);
			
			\draw [arrow] (estimateCj) -- (jis1);
			\draw [arrow] (estimateCj) -- (strongmax);
			
			\draw [arrow] (estimateCj) -- (viasquare);
			\draw [arrow] (viasquare) -- (l4);
			\draw [arrow] (viasquare) -- (pnot4);
			
		\end{tikzpicture}
	}
	\caption{\label{figure1} Scheme of proof of Theorem~\ref{main theorem}}
\end{figure}
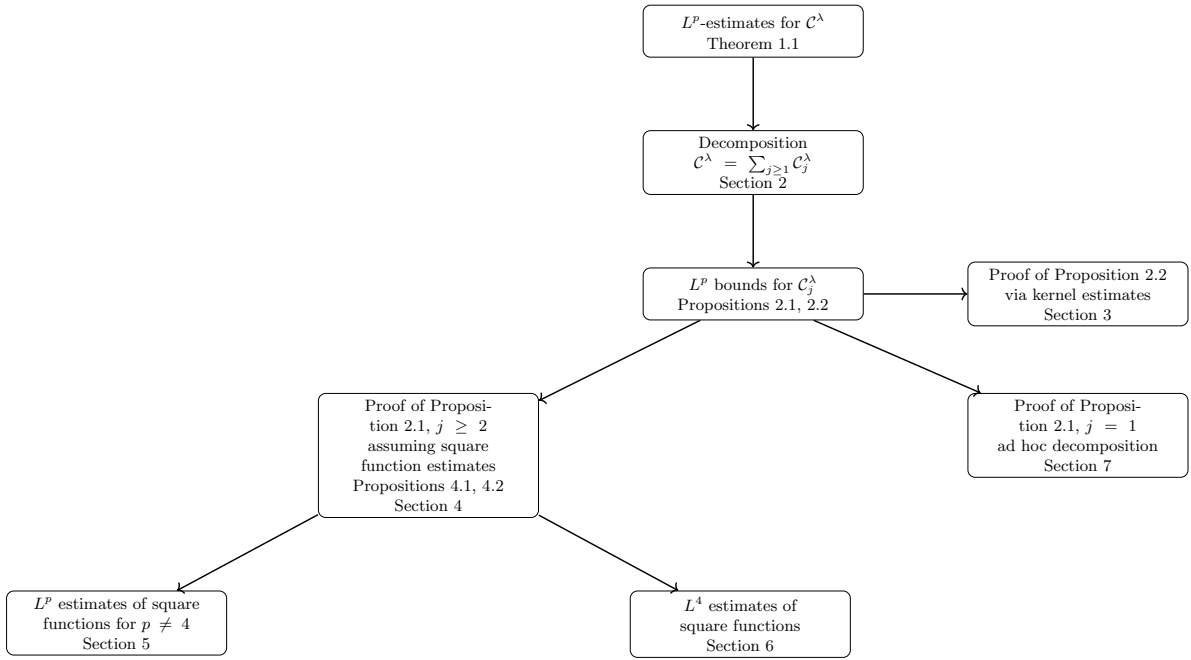

\subsection*{Notation}
Throughout the paper, we use standard notation. We write $A\lesssim B$ to indicate that $A\leq C B$ for some constant $C>0$ independent of the main parameters, and $A\approx B$ if both $A\lesssim B$ and $B\lesssim A$ hold. The value of $C$ may change from line to line.

\section{Decomposition of the multiplier. Proof of Theorem \ref{main theorem}}{\label{Decomposition}}

Consider $\xi=(\xi_1,\xi_2)\in \R^2$ and $\eta=(\eta_1,\eta_2)\in \R^2$, and let $\varphi \in C_{c}^{\infty}([1/2,2])$. Define
\begin{equation}
\label{eq:mlambda}
m^{\lambda}(\xi,\eta)=\Big(1-\frac{ \xi_1^{2}}{\xi_2^2}-\frac{ \eta_1 ^{2}}{\eta_2^2}\Big)^{\lambda}_{+}\varphi(\xi_2)\varphi(\eta_2).
\end{equation}
Let $\psi\in C^{\infty}_{c}([1/2,2])$ and $\psi_{1}\in C^{\infty}_{c}([0,3/4])$ be such that 
\begin{equation}
\label{eq:partition}
\sum_{j\geq2}\psi(2^{j}(1-t))+\psi_{1}(t)=1,\quad  t\in [0,1].
\end{equation}
We first decompose the bilinear multiplier $m^{\lambda}$ using a partition of unity in the $\xi$-variable. We write
\begin{equation}
\label{eq:decomp}
	m^{\lambda}(\xi,\eta)=\sum_{j\geq2}\psi\Big(2^{j}\big(1-\frac{\xi_1^2}{\xi^2_2}\big)\Big)m^{\lambda}(\xi,\eta)+\psi_{1}\Big(\frac{\xi_1^2}{\xi^2_2}\Big)m^{\lambda}(\xi,\eta)=:\sum_{j\geq 1}m^{\lambda}_{j}(\xi,\eta),
\end{equation}
where 
\begin{equation}
\label{eq:mlambdaj}
	m^{\lambda}_{j}(\xi,\eta)
	=\psi\Big(2^{j}\big(1-\frac{\xi_1^2}{\xi^2_2}\big)\Big)m^{\lambda}(\xi,\eta), \quad j\geq 2,
\end{equation}
and
\begin{equation}
\label{eq:ml1}
m^{\lambda}_{1}(\xi,\eta )=\psi_{1}\Big(\frac{\xi_1^2}{\xi^2_2}\Big)m^{\lambda}(\xi,\eta).
\end{equation}
Let $\mathcal C^{\lambda}_{j}$ denote the bilinear multiplier operator associated with $m^{\lambda}_{j}$; that is,   
\begin{equation}
\label{eq:Cljoriginal}
\mathcal C^{\lambda}_{j}(f,g)(x):=\int_{\mathbb{R}^{4}}m^{\lambda}_{j}(\xi,\eta) \widehat{f}(\xi)\widehat{g}(\eta)e^{2\pi i x\cdot(\xi+\eta)}~d\xi\, d\eta,\quad  j\geq 1.
\end{equation}
Observe that, in view of the bilinear interpolation and summation in $j$, Theorem~\ref{main theorem} reduces to establishing suitable estimates for the operators $\mathcal C_j^{\lambda}$. We begin by stating the main bounds that will imply Theorem {\ref{main theorem}}: items (i), (ii), and (iii) will follow from the following proposition.
\begin{proposition}{\label{T_j estimate}}
Let $2 \le p_1,p_2 \le \infty$ and $1 \le p \le \infty$ satisfying
$\frac1p=\frac1{p_1}+\frac1{p_2}$. 
Then for $j\ge1$,
\[
\|\mathcal C_j^{\lambda}(f,g)\|_{L^p}
\lesssim 2^{-\varepsilon  j} \|f\|_{L^{p_1}}\|g\|_{L^{p_2}},
\]
holds for every  $f\in L^{p_1}(\R^2)$ and $g\in L^{p_2}(\R^2)$, and for some $\varepsilon>0$,
in the following cases:
\begin{enumerate}
\item If $(p_1,p_2)\in \{(2,2), (2,4), (4,2),(4,4)\}$ and $\lambda>0$; 
\label{main estimate 1}
\item If $2 \le p_- \le 4 \le p_+ \le \infty$ and $
\lambda > \frac12-\frac{2}{p_+}$;
\label{main estimate 4}
\item If $4 \le p_1, p_2 \le \infty$ and $\lambda > 1-\frac{2}{p}$.
\label{main estimate 5}
\end{enumerate}
\end{proposition}
In turn, to achieve (iv) in Theorem  \ref{main theorem}, the task will boil down to proving the following proposition. 
\begin{proposition}{\label{T_j estimate 2}}
Let $1 < p_1, p_2 \le \infty$ and $1 \le p \le \infty$ satisfying $\frac{1}{p} = \frac{1}{p_1}+ \frac{1}{p_2}$, then for $\lambda>1$ and $j \ge 1$,
$$
\|\mathcal C_j^{\lambda}(f,g)\|_p \lesssim 2^{-\varepsilon j }\|f\|_{p_1}\|g\|_{p_2}
$$
holds for every $f \in L^{p_1}(\R^2)$ and $ g\in L^{p_2}(\R^2)$, and for some $\varepsilon>0$.
\end{proposition}

\begin{proof}[Proof of Theorem \ref{main theorem}]
    The proof of Theorem \ref{main theorem} (i), (ii), and (iii) is immediate after summing in $j$ the estimates obtained in Proposition \ref{T_j estimate}. Theorem \ref{main theorem} (iv) follows from linear interpolation between the estimates in Theorem \ref{main theorem}  (i), (ii), and in Proposition \ref{T_j estimate 2}. Indeed, the scheme is as follows:  we interpolate the estimates for  $(p_1,p_2)$ close to $(\infty,1)$, in Proposition \ref{T_j estimate 2} (after summation in $j$), with the estimates for $(p_1,2)$, $(2\le p_1\le 4)$ in Theorem~\ref{main theorem} (i), and with the estimates for $(p_1,2)$, $(4\le p_1 \le
    \infty)$ in Theorem~\ref{main theorem} (ii) (in the latter case, $p_+=p_1$). Analogous interpolation, with the roles of $p_1$ and $p_2$ reversed, yields the boundedness for points in the region symmetric with respect to the diagonal $p_1=p_2$.
\end{proof}

As we have seen above, the proof of Theorem \ref{main theorem} reduces to proving Propositions \ref{T_j estimate} and \ref{T_j estimate 2}; hence, the rest of the paper is devoted to the proof of these propositions. 
We first consider Proposition \ref{T_j estimate} for the case $j \ge 2$ and Proposition \ref{T_j estimate 2} for the case $j \ge 1$. The case $j=1$ of Proposition \ref{T_j estimate} will be dealt with separately in Section \ref{proofofj=1}. For the sake of simplicity, we can unify the notation for the cases $j=1$ and $j \ge 2$. Choosing $\widetilde{\psi}_1 \in C_c^{\infty}([1/2,2])$ appropriately such that $\psi_1(t) = \widetilde{\psi}_1(2(1-t))$, we have $$m^{\lambda}_{1}(\xi,\eta) = \widetilde{\psi}_1\Big(2\big(1-\frac{\xi_1^2}{\xi^2_2}\big)\Big)m^{\lambda}(\xi,\eta).$$
Then, by abuse of notation, we write $\widetilde{\psi}_1$ as $\psi$. Writing like this will allow us to manipulate $\mathcal C^{\lambda}_{j}$ for all $j \ge 1$ simultaneously in several parts of the manuscript.

 We make a further manipulation of $\mathcal C^{\lambda}_{j}$ which allows us to express the multiplier $m^{\lambda}_{j}$ as an integral involving two further operators.
Recall the identity, see \cite[p.~278]{SW}:
\begin{align} {\label{Stein-Weiss inequality}}
\big(R^2-|m|^2\big)^{\lambda}=c_{\mu, \nu} \int^{R}_{|m|}(R^2-t^2)^{\mu-1} t^{2\nu+1} \Big(1-\frac{|m|^2}{t^2}\Big)^{\nu}~dt,
	\end{align}
where $\lambda>0$, $\lambda=\mu+\nu$ with $\mu>0$ and $\nu>-1$, and $c_{\mu, \nu}=\frac{2\Gamma(\mu+\nu+1)}{\Gamma(\nu+1)\Gamma(\mu)}$. 
We apply \eqref{Stein-Weiss inequality} with 
$$
R^2=1-\frac{\xi_1^2}{\xi_2^2}\quad \text{and} \quad |m|^2=\frac{ \eta_1 ^{2}}{\eta_2^2} 
$$
and, in view of \eqref{eq:mlambda}, \eqref{eq:decomp} and \eqref{eq:mlambdaj}, this yields
\begin{align*}
m^{\lambda}_{j}(\xi,\eta )
	&=c_{\mu, \nu} \varphi(\xi_2)\varphi(\eta_2)\psi\Big(2^{j}\big(1-\frac{\xi_1^2}{\xi^2_2}\big)\Big)\int^{\sqrt{2^{1-j}}}_{0}\Big( 1-\frac{\xi_1^2}{\xi^2_2}-t^2\Big)^{\mu-1}_{+} \Big(1-\frac{\eta_1^2}{t^2\eta^2_2}\Big)^{\nu}_{+}~t^{2\nu+1}dt,
\end{align*}
where $\lambda=\mu+\nu$ with $\mu>0$, $\nu>-1$. Such a representation motivates the introduction of the following operators. For $j\ge 1$ and $t>0$, define 
	\begin{equation}
    \label{eq:Toperator}
	T^{\nu}_{t}g(x)=\int_{\mathbb{R}^{2}}\Big(1-\frac{\eta_1^2}{t^2\eta^2_2}\Big)^{\nu}_{+}\varphi(\eta_{2}) \widehat{g}(\eta) e^{2\pi i x\cdot\eta} d\eta
\end{equation}
and 
\begin{equation}
\label{eq:Bjt}
	B^{\mu}_{j,t}f(x)=\int_{\mathbb{R}^2}\psi\Big(2^{j}\big(1-\frac{\xi_1^2}{\xi^2_2}\big)\Big)\varphi(\xi_2)\Big(1-\frac{\xi_1^2}{\xi^2_2}-t^2\Big)^{\mu-1}_{+}\widehat{f}(\xi) e^{2\pi i x\cdot\xi} d\xi.
	\end{equation}
Observe that the operator $T^{\nu}_{t}$ does not depend on $j$.
With this notation, the multiplier $m_j^{\lambda}$ is expressed, for $j\ge1$, as a superposition (in $t$) of products of linear cone-type multipliers, namely
\begin{equation}
\label{eq:twoperators}
\mathcal C^{\lambda}_{j}(f,g)(x)= c_{\mu,\nu} \int^{\sqrt{2^{1-j}}}_{0} T^{\nu}_{t}g(x)t^{2\nu+1}B^{\mu}_{j,t}f(x)\,dt.
\end{equation}

The proof of Proposition~\ref{T_j estimate 2} relies on a domination criterion via strong maximal function. On the other hand, the proof of Proposition~\ref{T_j estimate} is deeper and we distinguish the cases $j=1$ and $j\ge2$. For the latter, we first reduce the analysis of $\mathcal{C}^{\lambda}_j$ to prove suitable square function estimates. Such estimates will be obtained by making use of different ingredients: $L^4$ estimates require a more refined analysis involving directional square functions and almost orthogonality estimates; $L^p$ estimates with $p>4$ will be treated by establishing suitable kernel estimates for the underlying operators. The domination by strong maximal functions will also play a relevant role in both cases. Finally, the case $j=1$ in Proposition~\ref{T_j estimate} requires another decomposition of the multiplier.

\section{Strong maximal function domination of linear cone-type multipliers}\label{sec:dominationbystrongmax}

In this section we will provide pointwise domination results by strong maximal functions for the operators $T^{\nu}_t$ and  $B^{\mu}_{j,t}$. These will be used in the proof of Proposition \ref{T_j estimate 2}, which is contained in this section, and they will play a crucial role in the proof of Proposition \ref{T_j estimate}.
For $t\ge0$, let us define  $\mathfrak{m}_tf(x)$ to be the analogue of the strong maximal function given by 
	\begin{equation}{\label{directional maximal function m_t}}
		\mathfrak{m}_tf(x):= \sup_{x \in R}\frac{1}{|R|}\int_{R}|f(y)| \, dy,
\end{equation}
where the supremum is taken over all rectangles whose longest side is in the direction $(1,t)$ or $(1,-t)$. Observe that 
$\mathfrak{m}_0=:\mathfrak{m}$ is exactly the usual strong maximal function. Concerning the operator $T^{\nu}_{t}$, we have the following pointwise domination.

\begin{proposition}\label{Tnu}
	Let $\nu>0$ and $0<t \le 1$. Then,  
	$$|T^{\nu}_tg(x) |\lesssim \mathfrak{m}_tg(x),$$
	where the implicit constant is independent of $t$. 
\end{proposition}
On the other hand, we have a pointwise estimate for $B^{\mu}_{j,t}$.
\begin{proposition}\label{Bmu}
	Let $j \geq 1$ and $\mu>1$. Then 
	$$|B^{\mu}_{j,t}f(x)|\lesssim (1-t^2)^{\mu-1} \mathfrak{m}_{\sqrt{1-t^2}}\circ \mathfrak{m}_1f(x),$$
	where the implicit constant is independent of $j$ and $t$. 
\end{proposition}
Propositions \ref{Tnu} and \ref{Bmu} will be enough to prove Proposition \ref{T_j estimate 2}. Nevertheless, for the proof of Proposition \ref{T_j estimate}, we need a finer result.
If $t\in [0,\sqrt{2^{-2-j}}]$ and $\mu>0$, then the multiplier of  $B^{\mu}_{j,t}$ is away from the singularity. In this case we have a stronger estimate. 
\begin{proposition}\label{operatorB}
	Let $j\geq 2$ and  $\mu>0$. Then for $t\in [0,\sqrt{2^{-2-j}}]$,
	$$|B^{\mu}_{j,t}f(x)|\lesssim 2^{(1-\mu)j} \mathfrak{m}_1f(x),$$ where the implicit constant is independent of $j$ and $t$. 
\end{proposition}

\subsection{Proof of Proposition \ref{T_j estimate 2}}

Fix $j\ge1$. Given $\lambda >1$, we consider $\mu>1$ and $\nu>0$ with $\lambda=\mu+\nu$. Applying Proposition {\ref{Tnu}} and Proposition {\ref{Bmu}}, we have  
$$ \left|\mathcal{C}_j^{\lambda}(f,g)(x)\right| = \left|c_{\mu,\nu} \int_{0}^{\sqrt{2^{1-j}}}B^{\mu}_{j,t}f(x)T_t^\nu g(x) t^{2 \nu+1}  \, dt \right|\lesssim \int_{0}^{\sqrt{2^{1-j}}}(1-t^2)^{\mu-1}\mathfrak{m}_{\sqrt{1-t^2}} \circ \mathfrak{m}_1f(x)\mathfrak{m}_t g(x) t^{2 \nu+1} \,dt.
$$ 
Consequently, in view of the boundedness of the strong maximal function, we get that  
\begin{align*}
\|\mathcal{C}_j^{\lambda}(f,g)\|_p  \lesssim \int_{0}^{\sqrt{2^{1-j}}}\|\mathfrak{m}_{\sqrt{1-t^2}} \circ \mathfrak{m}_1f \cdot \mathfrak{m}_t g\|_p t^{2\nu+1}  \, dt   &\le \int_{0}^{\sqrt{2^{1-j}}}\|\mathfrak{m}_{\sqrt{1-t^2}} \circ \mathfrak{m}_1f\|_{p_1} \|\mathfrak{m}_t g\|_{p_2} t^{2\nu+1}  \, dt \\
& \lesssim 2^{-j(\nu+1)}\|f\|_{p_1}\|g\|_{p_2}.
\end{align*}

\subsection{Proof of Proposition~\ref{Tnu}: domination of $T^{\nu}_t g$ by strong maximal function}\label{kernelestimates}

Recall the partition of unity in \eqref{eq:partition}. Then we have the following decomposition of the operator
\begin{equation}\label{partitionofunity}|{T}^{\nu}_tg(x)|\leq |T^{\nu,\psi_1}_{t}g(x)|+\sum_{\ell \ge 2}  2^{-\ell\nu}|T^{\nu,\psi}_{2^{-\ell},t}g(x)|,\quad x\in \R^2,
	\end{equation}
where
\begin{equation}
\label{Tuno}
T^{\nu,\psi_1}_{t}g(x)=\int_{\mathbb{R}\times [1/2,2]} \varphi(\eta_2)\psi_1\Big(\frac{\eta^2_1}{t^2\eta^2_2}\Big) \Big(1-\frac{\eta_1^2}{t^2\eta_2^2}\Big)^{\nu}\widehat{g}(\eta)e^{2\pi i x \cdot \eta} d\eta
\end{equation}
and for $2^{-\ell}=\delta$ we write 
\begin{equation*}
\label{def:Tnudelta}
T^{\nu,\psi}_{\delta,t}g(x)=\int_{\mathbb{R}\times [1/2,2]} \varphi(\eta_2)\psi\Big(\delta^{-1}\big(1-\frac{\eta^2_1}{t^2\eta^2_2}\big)\Big) \Big(\delta^{-1}\big(1-\frac{\eta^2_1}{t^2\eta^2_2}\big)\Big)^{\nu}\widehat{g}(\eta)e^{2\pi i x \cdot \eta} d\eta.
\end{equation*}
Observe that $T^{\nu,\psi_1}_{t}$ and $T^{\nu,\psi}_{\delta,t}$ represent the non-singular and the singular part of $T^{\nu}_t$, respectively.

We first show that, for $\delta < \frac{1}{2}$,
\begin{equation}{\label{Tdelta}}
|T^{\nu,\psi}_{\delta,t}g(x)|\lesssim \mathfrak{m}_tg(x),  
\end{equation}
where the implicit constant is independent of $\delta$ and $t$. Note that the kernel of $T^{\nu,\psi}_{\delta,t}$ is given by 
\begin{align}
K^{\nu,\psi}_{\delta,t}(y) & = \int_{\mathbb{R}\times [1/2,2]} \varphi(\eta_2) \psi\Big(\delta^{-1}\big(1-\frac{\eta^2_1}{t^2\eta^2_2}\big)\Big) \Big(\delta^{-1}\big(1-\frac{\eta^2_1}{t^2\eta^2_2}\big)\Big)^{\nu}e^{2\pi i y\cdot \eta} d\eta \nonumber \\
& = t\int^{2}_{1/2} \eta_2\varphi(\eta_2)  e^{2\pi i\eta_2 y_2}  \int_{\mathbb{R}}\psi \big(\delta^{-1}(1-\eta^2_1)\big)  \big(\delta^{-1}(1-{\eta^2_1})\big)^{\nu}e^{2\pi i t\eta_2 y_1\eta_1}\,d\eta_1 d\eta_2. \label{a middle step in the maximal domination}  
\end{align}
We decompose the integral over $\R$ into two parts, $[0,\infty)$ and $(-\infty, 0]$, obtaining two kernels 
$$
(K^{\nu,\psi}_{\delta,t})_1(y)+(K^{\nu,\psi}_{\delta,t})_2(y)
:=t\int^{2}_{1/2} \eta_2\varphi(\eta_2)  e^{2\pi i\eta_2 y_2} \Big(\int_{0}^{\infty}+\int_{-\infty}^{0}\Big)\Psi \big(\delta^{-1}(1-\eta^2_1)\big)  e^{2\pi i t\eta_2 y_1\eta_1}\,d\eta_1 d\eta_2
$$
where $\Psi(\eta) = \psi(\eta)\eta^{\nu}$. Note that, by symmetry, $(K^{\nu,\psi}_{\delta,t})_2(y_1,y_2) = (K^{\nu,\psi}_{\delta,t})_1(-y_1,y_2)$, so we only need to consider the kernel $(K^{\nu,\psi}_{\delta,t})_1$. After a change of variable we write
\begin{equation}
\label{eq:K1change}
(K^{\nu,\psi}_{\delta,t})_1(y) = t\int^{2}_{1/2} \eta_2\varphi(\eta_2)  e^{2\pi i \eta_2( ty_1+y_2)} \Phi_{\delta^{-1}}(t,y_1,\eta_2)\,d\eta_2, 
\end{equation}
where $\Phi_{\delta^{-1}}(t,y_1,\eta_2)=\int_{-\infty}^{1}\Psi(\delta^{-1}(2\eta_1-\eta^2_1))  e^{-2\pi i t\eta_2 \eta_1 y_1} \,d\eta_1$. 

We claim that
\begin{equation}{\label{K1}}
|(K^{\nu,\psi}_{\delta,t})_1(y)|\leq C_{\nu, \psi} ~t\delta (1+ t \delta|y_1|)^{-3}\min\{ 1, |ty_1+y_2|^{-3} \}.    
\end{equation}
To prove the claim, first we observe that  $\Psi(\eta)$ is a bounded function and, in view of the support of $\psi$,
\begin{equation}\label{condition for the support}
\{\eta_1 \in (-\infty,1]: \Psi(\delta^{-1}(2\eta_1-\eta_1^2)) \neq 0\} \subset [0, 2\delta].    
\end{equation}
Using this  and integration by parts with respect to $\eta_1$, it can be checked that $|\frac{d^n}{d\eta^n_2}\Phi_{\delta^{-1}}(t,y_1,\eta_2)|\lesssim \delta (1+t \delta|y_1|)^{-N}$ for any $n \geq 0$, $N>0$, uniformly in  $\eta_2\in [1/2,2]$. Then,
$$
|(K^{\nu,\psi}_{\delta,t})_1(y)| \le t\int^{2}_{1/2} |\varphi(\eta_2)||\Phi_{\delta^{-1}}(t,y_1,\eta_2)| d\eta_2 \lesssim  t\delta (1+t \delta|y_1|)^{-3}.$$
For the second part of the claim, applying integration by parts with respect to $\eta_2$,    
\begin{align*}
|(K^{\nu,\psi}_{\delta,t})_1(y)|& = t \Big|\int^2_{1/2} \frac{d^3}{d\eta^3_2} (\varphi(\eta_2) \Phi_{\delta^{-1}}(t,y_1,\eta_2))  \frac{e^{2\pi i \eta_2 (ty_1+y_2)}}{[2\pi i (ty_1+y_2)]^3}~d\eta_2\Big|\\
&\lesssim t\delta (1+t \delta|y_1|)^{-3}|ty_1+y_2|^{-3},
\end{align*}            
hence we prove the claim.

In view of the kernel estimates above, in order to prove the desired estimate \eqref{Tdelta} it is enough to work with $y\in \mathfrak{Q}_4=\{y: y_1 \geq 0,y_2 \leq 0\}$. The case of other three quadrants can be handled similarly. We decompose the quadrant $\mathfrak{Q}_4$ as  follows
$$
\R^2\cap \mathfrak{Q}_4=\Big(\bigcup_{i \ge 1}A_i \Big) \cup A_0, 
$$
where 
$$
A_0 := B(0, (\delta t)^{-1})\cap \mathfrak{Q}_4 \quad \text{and}\quad 
A_i :=\Big(B(0,(\delta t)^{-1}2^{i})\setminus B(0,(\delta t)^{-1}2^{i-1})\Big)\cap \mathfrak{Q}_4,\quad i\geq 1.
$$
For $i \geq 0$, we further decompose $A_i$ into  $- \log(\delta t) $ almost rectangles $\{A_{i,k}\}_{k=0}^{\lfloor - \log(\delta t) \rfloor+1}$, where
\begin{align*}
A_{i,0}&:= A_i \cap \{y \in \R^2 : |ty_1 + y_2| < 2^{i-2}\},\\
A_{i,k}&:= A_i \cap \{y \in \R^2 : 2^{i+k-3} \le |ty_1 + y_2| < 2^{i+k-2}\},\quad \text{ for } 1 \le k \le \lfloor - \log(\delta t) \rfloor ,\\
A_{i,\lfloor - \log(\delta t)\rfloor+1}&:= A_i \cap \{y \in \R^2 :  (\delta t)^{-1}2^{i-2} \le |ty_1 + y_2|\}.
\end{align*}
See Figure \ref{fig:Ai}.

\begin{center}
\begin{figure}
\begin{tikzpicture}[scale=0.75]
\draw[->] (-0.2,0) --(6.5,0) node[right] {$y_1$};
\draw[->] (0,0.2) --(0,-6.5) node[below] {$y_2$};
\draw (0,-6) circle arc (270:360:6);
\filldraw (0,-6) circle (.05);
\filldraw (0,-6) node[left] {$-(\delta t)^{-1}2^{i}$};
\draw (0,-3) circle arc (270:360:3);
\filldraw (0,-3) circle (.05);
\filldraw (0,-3) node[left] {$-(\delta t)^{-1}2^{i-1}$};
\draw[dashed, domain=(2-2*sqrt(11))/5:(4-sqrt(716))/10, smooth, variable=\t] plot ({1-2*\t}, {\t});
\draw[dashed, domain=(-2-2*sqrt(11))/5:(-4-sqrt(716))/10, smooth, variable=\t] plot ({-1-2*\t}, {\t});
\node[rotate=-26.565] at (4.024926,-2.012463) {$A_{i,0}$};
\draw[dashed, domain=(8-sqrt(164))/10:(8-sqrt(704))/10, smooth, variable=\t] plot ({2-2*\t}, {\t});
\draw[dashed, domain=(-8-sqrt(164))/10:(-8-sqrt(704))/10, smooth, variable=\t] plot ({-2-2*\t}, {\t});
\node[rotate=-26.565] at (4.2799,-1.3900) {$A_{i,1}$};
\node[rotate=-26.565] at (3.6799,-2.5900) {$A_{i,1}$};
\draw[dashed, domain=(-16-sqrt(116))/10:(-16-sqrt(656))/10, smooth, variable=\t] plot ({-4-2*\t}, {\t});
\node[rotate=-26.565] at (3.2419, -3.1209) {$A_{i,2}$};
\end{tikzpicture} 
\caption{Decomposition of $A_i$}
\label{fig:Ai}
\end{figure}
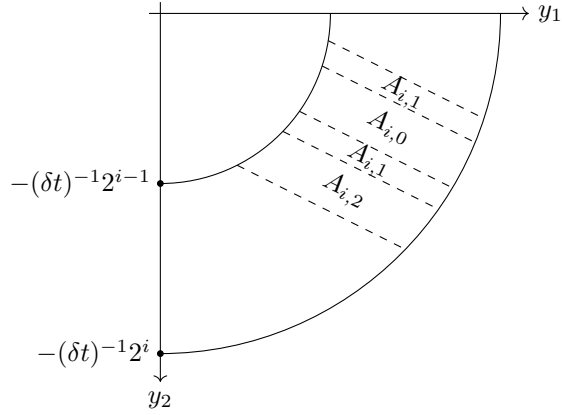
\end{center}

For each $0 \le k \le \lfloor - \log(\delta t)\rfloor+1$, we can find rectangles $R_{i,k}$ satisfying $R_{i,k} \supset A_{i,k}$ and $0 \in R_{i,k}$, with longest side in the direction $(1, -t)$ and dimension $100(\delta t)^{-1}2^{i} \times 100(2^{i+k})$. 
Then for $i \ge 1$,
\begin{align}
\int_{A_i}|(K^{\nu,\psi}_{\delta,t})_1(y)g(x-y)| \, dy &\lesssim t\delta\sum\limits_{k=0}^{\lfloor  -\log(\delta t)\rfloor+1}\int_{A_{i,k}}\frac{|g(x-y)|}{(1+\delta t|y_1|)^{3}}\min\{ 1, (ty_1+y_2)^{-3}\} \, dy\notag \\
& \lesssim t \delta \Big(\int_{A_{i,0}}\frac{|g(x-y)|}{2^{3i}} \, dy + \sum\limits_{k=1}^{\lfloor -\log(\delta t)\rfloor + 1}\int_{A_{i,k}} \frac{|g(x-y)|}{ 2^{3(i+k)}} \, dy \Big)\notag \\
& \lesssim \frac{2^{-i}}{|R_{i,0}|}\int_{R_{i,0}}|g(x-y)| \, dy +\sum\limits_{k=1}^{\lfloor -\log(\delta t)\rfloor + 1}  \frac{2^{-i-2k}}{|R_{i,k}|}\int_{R_{i,k}}|g(x-y)| \, dy {\label{intermediate result}} \\
&\notag \lesssim 2^{-i}\mathfrak{m}_tg(x), 
\end{align}
where we use the fact that $|y_1| \simeq (\delta t)^{-1}2^{i}$ for $y \in A_{i,0}$ and $|ty_1 + y_2| \simeq 2^{i+k}$ for $y \in A_{i,k}$ with $k \ge 1$. Similarly, we have 
$$
\int_{A_0}|(K^{\nu,\psi}_{\delta,t})_1(y)g(x-y)| \, dy \lesssim \mathfrak{m}_tg(x),
$$
where the implicit constant is independent of $\delta$ and $t$. 
Consequently, we get \eqref{Tdelta}. 

It remains to prove
\begin{equation}
\label{estimate of inner part psi_1}
|T^{\nu,\psi_1}_{t}g(x)| \lesssim \mathfrak{m}_tg(x). 
\end{equation} 
The proof of estimate \eqref{estimate of inner part psi_1} follows a similar strategy as that of estimate \eqref{Tdelta}. The kernel $K_{t}^{\nu,\psi_1} $associated with \eqref{Tuno} is given by 
\begin{align*}
K_{t}^{\nu,\psi_1}(y)&=\int_{\mathbb{R}\times [1/2,2]}\varphi(\eta_2)\psi_1\Big(\frac{\eta^2_1}{t^2\eta^2_2}\Big) \Big(1-\frac{\eta_1^2}{t^2\eta_2^2}\Big)^{\nu}e^{2\pi i x \cdot \eta} d\eta\\
&= t\int^{2}_{1/2} \eta_2\varphi(\eta_2)  e^{2\pi i\eta_2 y_2}  \int_{\mathbb{R}}\psi_1(\eta_1^2)  (1-{\eta^2_1})^{\nu}e^{2\pi i t\eta_2 y_1\eta_1}\,d\eta_1 d\eta_2.
\end{align*}
In this case we do not decompose $K_{t}^{\nu,\psi_1}$ into two parts: if we did that, endpoint terms would artificially appear while applying integration by parts. Instead, we directly perform a change of variable  $1-\eta_1 \mapsto \eta_1$ and obtain estimate \eqref{K1} with $\delta=1$ using integration by parts. Thus, we can proceed as before and obtain the estimate \eqref{estimate of inner part psi_1}.
\qed

We remark that the assumption $\nu > 0$ in Proposition~\ref{Tnu} is used for the summation over $\ell$. For a fixed $\delta=2^{-\ell}$, we have the following result, which holds for all $\nu > -1$, and which is obtained by inspection of the proof of Proposition \ref{Tnu}, from estimate \eqref{intermediate result} and the corresponding ones for $A_0$ and the other quadrants. It will play a crucial role in the proof of the $L^4$ estimates in Section \ref{sec:L4sf}.
\begin{coro}{\label{eccentricity result}}
Suppose $\nu > -1$. There exists a sequence $\{C_i\}_{i=1}^{\infty}$ with $\|\{C_i\}\|_{\ell^{1}} = 1$, such that for each $0< \delta < \frac{1}{2}$, $0< t <1$, we can find a family of rectangles $\{R_{i}^{t}\}$ whose eccentricity is between $1$ and $(\delta t)^{-1}$ satisfying
$$
|T^{\nu,\psi}_{\delta,t}g(x)| \lesssim \sum\limits_{i \ge 1}C_i\frac{\chi_{R_i^t}}{|R_i^t|} * |g|(x),$$
where the implicit constant is independent of $\delta$ and $t$.
\end{coro}

\subsection{Proof of Proposition~\ref{Bmu}: domination of $B^{\mu}_{j,t}f$ by strong maximal function for $\mu > 1$}

Fix $j\ge1$. The idea of this proof is a variant of the Proposition \ref{Tnu}. Let $\widetilde{\varphi}(x)$ be a smooth function supported in $[1/4,4]$ which takes value $1$ on $[1/2,2]$.   
We rewrite $B^{\mu}_{j,t}f(x)$ as follows. 
\begin{align*}
B^{\mu}_{j,t}f(x)=&\int_{\mathbb{R} \times [1/2,2]} \Big(1-t^2-\frac{\xi^2_1}{\xi^2_2}\Big)^{\mu-1}_{+}\varphi(\xi_2)  \widetilde{\varphi}(\xi_2) \psi\Big(2^j\big(1-\frac{\xi^2_1}{\xi^2_2}\big)\Big) \widehat{f}(\xi) e^{2\pi i x\cdot\xi} \,d\xi\\
=& (1-t^2)^{\mu-1} \int_{\mathbb{R} \times [1/2,2]} \Big(1-\frac{\xi^2_1}{(1-t^2)\xi^2_2}\Big)^{\mu-1}_{+} \varphi(\xi_2) \widehat{f_{1,j}}(\xi) e^{2\pi i x\cdot\xi} \,d\xi\\
=&(1-t^2)^{\mu-1} T^{\mu-1}_{\sqrt{1-t^2}}f_{1,j}(x),
\end{align*}
where $T^{\mu-1}_{\sqrt{1-t^2}}$ is the operator defined in \eqref{eq:Toperator} and
\begin{equation}{\label{fRj}}
\widehat{f_{1,j}}(\xi)= \widetilde{\varphi}(\xi_2) \psi\Big(2^j\big(1-\frac{\xi^2_1}{\xi^2_2}\big)\Big) \widehat{f}(\xi).  
\end{equation}
Invoking Proposition {\ref{Tnu}} we have, for $\mu>1$, 
$$
| B^{\mu}_{j,t}f(x)|\leq (1-t^2)^{\mu-1} |T^{\mu-1}_{\sqrt{1-t^2}}f_{1,j}(x)| \lesssim (1-t^2)^{\mu-1}|\mathfrak{m}_{\sqrt{1-t^2}}f_{1,j}(x)|.
$$
Then, to complete the proof, we need to show that     
\begin{equation}\label{fj}
|f_{1,j}(x)|\lesssim \mathfrak{m}_1f(x).
\end{equation}

We write $f_{1,j}(x)=K_{1,j} \ast f(x)$, where 
\begin{align*}
K_{1,j}(y)&=\int_{\mathbb{R}\times [1/4,4]} \widetilde{\varphi}(\xi_2) \psi\Big(2^{j}\big(1-\frac{\xi^2_1}{\xi^2_2}\big)\Big) e^{2\pi i y\cdot\xi}\,d\xi\\
&=\int^4_{1/4} \widetilde{\varphi}(\xi_2) \xi_2 e^{2\pi i \xi_2(y_2+y_1)} \int_{\mathbb{R}} \psi(2^j(2 \xi_1 - \xi_1^2)) e^{-2\pi i  \xi_2 \xi_1 y_1}\,d\xi_1 d\xi_2,
\end{align*}
so we arrive at essentially the same position as in \eqref{eq:K1change}  with $\delta = 2^{-j}$ and $t = 1$. Therefore, continuing likewise, we obtain estimate \eqref{fj}.

\subsection{Proof of Proposition {\ref{operatorB}}: domination of $B^{\mu}_{j,t}$ away from the singularity by strong maximal function for $\mu>0$}
Fix $j\ge 2$. Observe that the kernel of the operator $B^{\mu}_{j,t}$  is given by 
	\begin{align*}
	K^{\mu}_{j,t}(y) &= \int_{\mathbb{R}^2} \psi\Big(2^j\big(1 - \frac{\xi_1^2}{\xi_2^2}\big)\Big)\varphi(\xi_2)\Big(1 - \frac{\xi_1^2}{\xi_2^2} - t^2\Big)^{\mu - 1}_+ e^{2\pi i y \cdot \xi} \, d\xi\\
	&= \int_{\mathbb{R}} \int_{\mathbb{R}} \psi(2^j(1 - u^2))\, \varphi(\xi_2)\, (1 - u^2 - t^2)_+^{\mu - 1} e^{2\pi i (y_1 u + y_2)\xi_2} \xi_2\, d\xi_2\, du,
	\end{align*}
where we recall that $y=(y_1, y_2)\in \R^2$, $\psi,\varphi \in C_c^\infty([1/2, 2])$, and $t \in [0,\, \sqrt{2^{-2-j}}]$. 

	 Since \( \operatorname{supp} \psi \subset [1/2, 2] \), in the integral above, we have 
     $$
     u \in\big[\sqrt{1-2^{-j+1}}, \sqrt{1-2^{-j-1}}\big] \cup\big[-\sqrt{1-2^{-j-1}}, -\sqrt{1-2^{-j+1}}\big]=:V_j\cup (-V_j).
     $$
Let us write
     $$
     K^{\mu,1}_{j,t}(y) := \int_{V_j} \int_{\mathbb{R}} \psi(2^j(1 - u^2))\, \varphi(\xi_2)\, (1 - u^2 - t^2)_+^{\mu - 1} e^{2\pi i (y_1 u + y_2)\xi_2} \xi_2\, d\xi_2\, du.
     $$
Observe that when $t\in [0,\sqrt{2^{-2-j}}]$, we have $1-u^2-t^2 \simeq 2^{-j}$, and thus the symbol $(1-u^2-t^2)^{\mu-1}_{+}$ does not have any singularity for every $\mu > 0$. We provide two different estimates for the kernel $K^{\mu,1}_{j,t}$.

On the one hand, let $\Psi_1(\xi_2):=\xi_2\varphi(\xi_2)$, then $\Psi_1\in C_c^\infty([1/2,2])$, and we have the following decay estimate
\begin{align*}
\Big|\int_{\mathbb{R}} \Psi_1(\xi_2)\, e^{2\pi i (y_1 u + y_2)\xi_2}\, d\xi_2\Big| 
	&\leq \frac{C_N}{(1+|y_1u+y_2|)^N},\quad  N\in \N. 
	\end{align*}
	Hence,
\begin{align}
	|K^{\mu,1}_{j,t}(y)| &\lesssim \int_{\mathbb{R}} |\psi(2^j(1 - u^2))|\, (1 - u^2 - t^2)^{\mu - 1} (1 + |y_1 u + y_2|)^{-N}\, du \nonumber \\
    &\lesssim  \sup_{u\in V_j}(1 + |y_1 u + y_2|)^{-N}  \int^{{\sqrt{1-2^{-j-1}}}}_{{\sqrt{1-2^{-j+1}}}}(1-u^2-t^2)^{\mu-1} \,du \nonumber \\
    & \lesssim 2^{-\mu j} \sup_{u\in V_j}(1 + |y_1 u + y_2|)^{-N}. \label{A1}
\end{align}

On the other hand, we can write
\begin{align*}
K^{\mu,1}_{j,t}(y) = \int^2_{1/2} \varphi(\xi_2) e^{2\pi i y_2\xi_2} \int_{V_j} \psi(2^j(1-u^2)) (1-u^2-t^2)^{\mu-1} e^{2\pi i \xi_2 y_1u} \,du\, d\xi_2  
\end{align*}
Observe that, for $u \in V_j$,
$$
    \Big|\frac{d^3}{du^3}(\psi(2^j(1-u^2)) (1-u^2-t^2)^{\mu-1})\Big|\lesssim  2^{-\mu j} 2^{4 j}. 
$$
Therefore, applying integration by parts with respect to $u$, we get that 
\begin{equation}\label{A2}
    |K^{\mu,1}_{j,t}(y)|\lesssim  2^{-\mu j} 2^{3j} (1+|y_1|)^{-3}. 
\end{equation}
We can proceed analogously with 
 $$
     K^{\mu,2}_{j,t}(y) := \int_{-V_j} \int_{\mathbb{R}} \psi(2^j(1 - u^2))\, \varphi(\xi_2)\, (1 - u^2 - t^2)_+^{\mu - 1} e^{2\pi i (y_1 u + y_2)\xi_2} \xi_2\, d\xi_2\, du.
     $$
Optimizing the estimates \eqref{A1} and \eqref{A2}, and the ones that can be obtained for $ K^{\mu,2}_{j,t}(y)$, we get
\begin{equation}
    |K^{\mu}_{j,t}(y)| \lesssim 2^{-\mu j} 2^{3j(1-\theta)}  (1+|y_1|)^{-3(1-\theta)} \label{finalkernelA}  \sup_{u\in V_j }\Big[(1+|uy_1-y_2|)^{-N}+(1+|uy_1+y_2|)^{-N}\Big]^{\theta}, 
\end{equation}
for any $\theta\in [0,1]$. 

It is enough to work with the kernel $K^{\mu}_{j,t}\chi_{\mathfrak{Q}_1}=:K^{\mathfrak{Q}_1}_{j,t}$, where $\mathfrak{Q}_1=\{y\in \R^2: y_1, y_2\geq 0\}$ denotes the first quadrant in $\R^2$. Decompose the kernel 
\begin{equation}{\label{decomposition of the kernel Kmuj,t}}
K^{\mathfrak{Q}_1}_{j,t}=\sum_{k\geq1} K^{\mathfrak{Q}_1}_{j,t,k}+K^{\mathfrak{Q}_1}_{j,t,0},    
\end{equation}
where 
    $K^{\mathfrak{Q}_1}_{j,t,0}:=K^{\mathfrak{Q}_1}_{j,t}\chi_{B(0,2^{j})}$ and $K^{\mathfrak{Q}_1}_{j,t,k}:=K^{\mathfrak{Q}_1}_{j,t}\chi_{B(0,2^{j+k})\setminus B(0,2^{j+k-1})}$.
Let us first consider the term $K^{\mathfrak{Q}_1}_{j,t,0}$. We decompose 
\[
B(0,2^j)\cap \mathfrak{Q}_1\subset \Big(\bigcup^j_{\ell=1}R_{j,\ell} \Big)\cup R_{j,0},
\]
        where $R_{j,\ell}:=\{ y\in B(0,2^j) \cap \mathfrak{Q}_1: \, 2^{\ell-1}\leq |y_1-y_2|\leq 2^{\ell}\}$ and $R_{j,0}:=\{ y\in B(0,2^j) \cap \mathfrak{Q}_1: |y_1-y_2|\leq1\}$. See Figure \ref{Decomposition of K_j,t,0}.

We first claim that for $y\in R_{j,\ell}$, $\ell\geq 0$, and $u\in V_j$,
\begin{equation}{\label{Kernel1}}
1+|uy_1-y_2|\gtrsim |y_1-y_2|,    
\end{equation}
where the implicit constant is independent of $u$. In the case of $\ell = 0 $ and $\ell=1$, the above estimate is trivial. Considering the case $\ell \ge 2$, note that 
\begin{align*}
    1+|uy_1-y_2|=1+|y_1-y_2-(1-u)y_1| \ge 1 + \big||y_1-y_2|-(1-u)y_1\big|,
\end{align*}
and for every $y\in R_{j,\ell}$ and $u\in V_j$ with $j \ge 2$, we have  $(1-u)y_1 \le \frac{4}{3}$ and $|y_1-y_2|\geq 2^{\ell-1}$, so we get \eqref{Kernel1}. The above claim implies that  
        \begin{equation}{\label{easypart}}
        \sup\limits_{u \in V_j}(1 + |y_1 u- y_2|)^{-N}\lesssim (|y_1-y_2|)^{-N}.
            \end{equation}
For each $R_{j,\ell}$, $0 \le \ell \le j$, we can choose a rectangle $\widetilde{R}_{j,\ell} \supset R_{j,\ell}$ containing the origin such that its longest side is in the direction $(1,1)$ and has size $|\widetilde{R}_{j,\ell}|\lesssim 2^{j+\ell}$. Applying (\ref{finalkernelA}) with $\theta=1$,
we get that 
\begin{align*}
    |K^{\mathfrak{Q}_1}_{j,t,0}\ast f(x)| 
    &\leq \sum^{j}_{\ell=1}\int_{R_{j,\ell}} |K^{\mathfrak{Q}_1}_{j,t}(y)f(x-y)| \,dy+\int_{R_{j,0}} |K^{\mathfrak{Q}_1}_{j,t}(y)f(x-y)| ~dy \\
& \lesssim \sum\limits_{\ell=1}^{j}2^{-\mu j} 2^{-\ell N}\frac{2^{j+\ell}}{|\widetilde{R}_{j,\ell}|}\int_{\widetilde{R}_{j,\ell}}|f(x-y)| \, dy  +2^{-\mu j} \frac{2^{j}}{|\widetilde{R}_{j,0}|}\int_{\widetilde{R}_{j,0}}|f(x-y)| \, dy  \\
 & \lesssim 2^{-\mu j} 2^{j} \mathfrak{m}_1f(x). 
    \end{align*}
\begin{center}
\begin{figure}
\begin{minipage}{0.495\linewidth}
\centering
\begin{tikzpicture}[scale=0.6]
\draw[->] (-0.2,0) --(6.5,0) node[right] {$y_1$};
\draw[->] (0,-0.2) --(0,6.5) node[above] {$y_2$};
\draw (6,0) circle arc (0:90:6);
\filldraw (0,6) circle (.05);
\filldraw (0,6) node[left] {$2^j$};
\draw[dashed, domain=0:3.851, smooth, variable=\t] plot ({\t}, {\t+0.75});
\draw[dashed, domain=0.75:4.601, smooth, variable=\t] plot ({\t}, {\t-0.75});
\node[rotate=45] at (2.8,2.8) {\scriptsize{$R_{j,0}$}};
\draw[dashed, domain=0:3.4258, smooth, variable=\t] plot ({\t}, {\t+1.5});
\node[rotate=45] at (2.2092,3.3342) {\scriptsize{$R_{j,1}$}};
\draw[dashed, domain=1.5:4.9258, smooth, variable=\t] plot ({\t}, {\t-1.5});
\node[rotate=45] at (3.3342,2.2092) {\scriptsize{$R_{j,1}$}};
\draw[dashed, domain=0:2.4685, smooth, variable=\t] plot ({\t}, {\t+3});
\node[rotate=45] at (1.6,3.85) {\scriptsize{$R_{j,2}$}};
\draw[dashed, domain=3:5.4685, smooth, variable=\t] plot ({\t}, {\t-3});
\node[rotate=45] at (3.85,1.6) {\scriptsize{$R_{j,2}$}};
\end{tikzpicture} 
\caption{\small{Decomposition of $K^{\mathfrak{Q}_1}_{j,t,0}$}}
\label{Decomposition of K_j,t,0}
\end{minipage}
\begin{minipage}{0.495\linewidth}
\centering
\begin{tikzpicture}[scale=0.6]
\draw[->] (-0.2,0) --(6.5,0) node[right] {$y_1$};
\draw[->] (0,-0.2) --(0,6.5) node[above] {$y_2$};
\draw (6,0) circle arc (0:90:6);
\filldraw (0,6) circle (.05);
\filldraw (0,6) node[left] {$2^{j+k}$};
\draw (3,0) circle arc (0:90:3);
\filldraw (0,3) circle (.05);
\filldraw (0,3) node[left] {$2^{j+k-1}$};
\filldraw (0,2.3) circle (.05);
\filldraw (0,2.3) node[left] {$2^{j+k-\frac{3}{2}}$};
\draw[dashed, domain=0:2.9338, smooth, variable=\t] plot ({\t}, {\t+2.3});
\node at (1,4.7) {$A_{j,k}^{\operatorname{off}}$};
\draw[dashed, domain=2.3:5.2338, smooth, variable=\t] plot ({\t}, {\t-2.3});
\node at (4.7,1) {$A_{j,k}^{\operatorname{off}}$};
\draw[dashed, domain=1.8:3.9320, smooth, variable=\t] plot ({\t}, {\t+0.6});
\draw[dashed, domain=2.4:4.5320, smooth, variable=\t] plot ({\t}, {\t-0.6});
\node[rotate=45] at (2.8,2.8) {\tiny{$R_{0}$}};
\draw[dashed, domain=1.4347:3.6, smooth, variable=\t] plot ({\t}, {\t+1.2});
\node[rotate=45] at (2.3424,3.2424) {\tiny{$R_{1}$}};
\draw[dashed, domain=2.6347:4.8, smooth, variable=\t] plot ({\t}, {\t-1.2});
\node[rotate=45] at (3.2424, 2.3424) {\tiny{$R_{1}$}};
\node[rotate=-45] at (3.8,3.8) {$A_{j,k}^{\operatorname{diag}}$};
\end{tikzpicture} 
\caption{\small{Decomposition of $K^{\mathfrak{Q}_1}_{j,t,k}$}}
\label{Decomposition of K_j,t,k}
\end{minipage}
\end{figure}
\end{center}

Next, for the terms $K^{\mathfrak{Q}_1}_{j,t,k}$ for $k \ge 1$, we proceed as follows 
$$
 	K^{\mathfrak{Q}_1}_{j,t,k}\ast f(x)
 	= \Big(\int_{A^{\operatorname{off}}_{j,k}} + \int_{A^{\operatorname{diag}}_{j,k}} \Big) K^{\mathfrak{Q}_1}_{j,t,k}(y)f(x-y)\, dy,
$$
where $A_{j,k}^{\operatorname{off}}:=\{y \in \mathfrak{Q}_1: 2^{j+k-1} \le |y| \le 2^{j+k}, |y_1-y_2| \ge 2^{j+k-\frac{3}{2}}\}$ and $A_{j,k}^{\operatorname{diag}}:=\{y \in \mathfrak{Q}_1: 2^{j+k-1} \le |y| \le 2^{j+k}, 0 \le |y_1-y_2| \le 2^{j+k-\frac{3}{2}}\}$. See Figure \ref{Decomposition of K_j,t,k}. Note that we can find a rectangle $\widetilde{R}_{j,k}^1 \supset A_{j,k}^{\operatorname{off}}$ containing the origin such that its longest side is in the direction $(1,1)$ and has size $|\widetilde{R}_{j,k}^1| \lesssim 2^{2(j+k)}$. Since for every $y \in A_{j,k}^{\operatorname{off}}$ and $u \in V_j$, we have $|y_1-y_2|\ge 2^{j+k-\frac{3}{2}}$ and $(1-u)y_1 \le \frac{4}{3} 2^k$, which imply estimate \eqref{easypart}. Thus, applying estimate (\ref{finalkernelA}) with $\theta=1$, we get that 
\begin{align*}\int_{A_{j,k}^{\operatorname{off}}}|K^{\mathfrak{Q}_1}_{j,t,k}(y)f(x-y)| \, dy 
	&\lesssim 2^{-\mu j}2^{-(j+k)N}\frac{2^{2(j+k)}}{|\widetilde{R}_{j,k}^1|}\int_{\widetilde{R}_{j,k}^1} |f(x-y)| \, dy \le 2^{-\mu j}2^{-(j+k)(N-2)}\mathfrak{m}_1f(x).
	\end{align*}
Summing over $k$ gives the desired result for this part.

Now we turn into the integral on $A_{j,k}^{\operatorname{diag}}$. A crucial observation is that for every $(y_1,y_2) \in A_{j,k}^{\operatorname{diag}}$, we have $|y_1| \ge 2^{j+k-3}$. We further decompose the integral as follows:
$$   
\Big|\int_{A^{\operatorname{diag}}_{j,k}} K^{\mathfrak{Q}_1}_{j,t,k}(y)f(x-y) dy\Big|\leq \sum^{j+k-1}_{\ell=1} \Big|\int K^{\mathfrak{Q}_1}_{j,t,k,\ell}(y)f(x-y)dy\Big|+\Big|\int K^{\mathfrak{Q}_1}_{j,t,k,0}(y) f(x-y)dy\Big|,
$$
where  $K^{\mathfrak{Q}_1}_{j,t,k,\ell}:=K^{\mathfrak{Q}_1}_{j,t}\chi_{R^{j+k}_{\ell}}$ with $R^{j+k}_{0}=\{y \in A_{j,k}^{\operatorname{diag}}:0 \leq |y_1-y_2|\leq 1\}$,  $R^{j+k}_{\ell}=\{y \in A_{j,k}^{\operatorname{diag}}: 2^{\ell-1}\leq |y_1-y_2|\leq 2^{\ell}\}$ for $1 \le \ell \le j+k-2$ and 
$R^{j+k}_{j+k-1}=\{y \in A_{j,k}^{\operatorname{diag}}: 2^{j+k-2}\leq |y_1-y_2|\leq 2^{j+k-\frac{3}{2}}\}$. See Figure \ref{Decomposition of K_j,t,k} (we remove the super-index $j+k$ of $R_{\ell}^{j+k}$ in the picture for the sake of simplicity). Observe that for each $0 \le \ell \le j+k-1$, we can find rectangles $\widetilde{R}^{j+k}_{\ell} \supset R^{j+k}_{\ell}$ with longest side aligned with the direction $(1,1)$ and $|\widetilde{R}^{j+k}_{\ell} |\lesssim 2^{j+k+\ell}$. 

We consider two cases. First, for $\ell \le k-6$ or $\ell \ge k+3$, we have $2^{-j-2}y_1 \ge 2|y_1-y_2|$ or $\frac{1}{2}|y_1-y_2| \ge  2^{-j+1}y_1$, respectively.
Moreover, for the latter cases  $2^{-j-2}y_1 \ge 2|y_1-y_2|$ or $\frac{1}{2}|y_1-y_2| \ge  2^{-j+1}y_1$, and $u\in V_j$, we have 
$$
    1+|uy_1-y_2| \ge 1+||y_1-y_2|-(1-u)y_1| \gtrsim 1+|y_1-y_2|.
$$
This gives us 
$$
\sup\limits_{u \in V_j}(1 + |u y_1 - y_2|)^{-N}\lesssim (1+|y_1-y_2|)^{-N}.
$$
 Therefore, applying the estimate (\ref{finalkernelA}) with $1/N < \theta < 2/3$, we get
\begin{align*}
    |K^{\mathfrak{Q}_1}_{j,t,k,\ell}\ast f(x)|&\lesssim 2^{-\mu j} 2^{3j(1-\theta)} 2^{-3(j+k)(1-\theta)} 2^{-\ell N\theta} \int_{\widetilde{R}^{j+k}_{\ell}} |f(x-y)|\,dy\lesssim 2^{-\mu j+j} 2^{-2k+3\theta k} 2^{\ell(1-N\theta)} \mathfrak{m}_{1}f(x). 
\end{align*}
Summing over $k$ and $\ell$, we obtain the desired result for this part. 

Secondly, for $k-6 \le \ell \le k+2$, we use the kernel estimate \eqref{finalkernelA} to deduce that 
\begin{align*}
   |K^{\mathfrak{Q}_1}_{j,t,k,\ell}\ast f(x)|&\lesssim  2^{-\mu j} 2^{3j(1-\theta)}  \int_{R^{j+k}_{\ell}}\frac{|f(x-y)|}{|y_1|^{3(1-\theta)}} dy \\
   &\lesssim 2^{-\mu j}2^{3j(1-\theta)}  2^{-(3j+3k)(1-\theta)}\int_{\widetilde{R}^{j+k}_{\ell}} |f(x-y)|\, dy\\
   &\leq  2^{-\mu j} 2^{j} 2^{-k+3\theta k}\mathfrak{m}_{1}f(x),
\end{align*}
where we used that $|\widetilde{R}^{j+k}_{\ell}|\lesssim 2^{j+k+\ell}\lesssim 2^{j+2k}$. Choosing $\theta<\frac{1}{3}$,  we sum over $k$ to complete the proof. 

\section{Square function estimates: proof of Proposition \ref{T_j estimate} for $j\ge2$}\label{sec:squarefunctiondomination}

In order to prove Proposition \ref{T_j estimate} for $j\ge 2$, we will first reduce the estimates for $\mathcal{C}_j^{\lambda}$ to estimates for certain square functions. Indeed, applying the Cauchy--Schwarz inequality in the $t$-variable, we obtain 
\begin{align}\label{decomintosquare}
\notag	|\mathcal C^{\lambda}_{j}(f,g)(x)|   \leq & c_{\mu,\nu} \Big(\int^{\sqrt{2^{1-j}}}_{0}|T^{\nu}_{t}g(x)|^{2}~t^{2b}dt\Big)^{1/2}\Big(\int^{\sqrt{2^{1-j}}}_{0}|B^{\mu}_{j,t}f(x)|^2~t^{2a}dt\Big)^{1/2} \\
    &=:c_{\mu,\nu}H^{\nu,j}_{b}g(x)G^{\mu,j}_{a}f(x), 
\end{align}
where the parameters $a$ and $b$ satisfy $2\nu+1=a+b$. The precise choice of $a$ and $b$ will be specified later, according to the range of exponents required in our estimates.

 In view of the preceding reduction, the bilinear estimate in Proposition \ref{T_j estimate} follows from suitable $L^p$ bounds for the square functions $G^{\mu,j}_{a}$ and $H^{\nu,j}_{b}$. These are contained in the propositions below.

\begin{proposition}
{\label{square function Gnu}}
Let $j \ge 2$ and $b>-1/2$. Then
$$
\|H^{\nu,j}_{b} g\|_p \lesssim 2^{-\varepsilon j} \|g\|_p,
$$
for all $ g \in L^{p}(\R^2)$ and for some $\varepsilon>0$, in the following cases:
\begin{enumerate}
\item If $p=2$, $\nu>-1/2$ {\label{Gnu estimate 1}}.
\item If $2< p \le \infty$, $\nu >0${\label{Gnu estimate 2}}.
\end{enumerate} 
In addition, for $\nu > -1/2$ and $k \in \N$, we have
\begin{equation}{\label{L4Gnuk}}
\Big\|\Big(\int^{2^{-k/2}}_{2^{-(k+1)/2}}|T^{\nu}_{t}g(x)|^2~t^{2b}dt\Big)^{1/2}\Big\|_4 \lesssim k\|g\|_4.
\end{equation}

\end{proposition}

\begin{proposition}
{\label{square function Gmu}}
Let $j \ge 2$. Then
$$
\|G^{\mu,j}_{a}f\|_p \lesssim 2^{-\varepsilon j} \|f\|_p, 
$$
for all $  f \in L^{p}(\R^2)$ for some $\varepsilon>0$, in the following cases:
\begin{enumerate}
\item If $p=2$, $\mu> 1/2$, $a>1/2$. {\label{Gmu estimate 1}}
\item If $2< p \le \infty$, $\mu>1$, $a> -1/2$. {\label{Gmu estimate 3}}
\item If $p=4$, $\mu> 1/2$, $a> 1/2$ {\label{Gmu estimate 2}}. 
\end{enumerate} 
In addition, for $2 \le p \le \infty$, $\mu > 0$, $k > 2+j$ and $a>-1/2$ we have
\begin{equation}{\label{LpGmuk}}
\Big\|\Big(\int^{2^{-k/2}}_{2^{-(k+1)/2}}|B^{\mu}_{j,t}f(x)|^2~t^{2a}dt\Big)^{1/2}\Big\|_p \lesssim 2^{-k(\frac{a}{2}+\frac{1}{4})}2^{(1-\mu)j}\|f\|_p. 
\end{equation}
\end{proposition}

\subsection{Proof of Proposition \ref{T_j estimate} for $j\geq 2$}

Along this subsection we will fix $j\ge2$.

\subsubsection{Proof of Proposition \ref{T_j estimate}, part \eqref{main estimate 1}} 

Fix $\lambda >0$ and take $\mu=a=\frac{1}{2}+\frac{\lambda}{2}$, $\nu=b=-\frac{1}{2}+\frac{\lambda}{2}$. Using H\"{o}lder's inequality we have that  
\[\|\mathcal C^{\lambda}_{j}(f,g)\|_{p} \lesssim \|H^{\nu,j}_{b}g\|_{p_2}\|G^{\mu,j}_{a}f\|_{p_1}.
\]
In view of the above estimate, observe that the  cases $(p_1,p_2)\in \{(2,2), (4,2)\}$ in part \eqref{main estimate 1} of Proposition~\ref{T_j estimate} follow by invoking \eqref{Gnu estimate 1} in Proposition \ref{square function Gnu}, and \eqref{Gmu estimate 1} and \eqref{Gmu estimate 2} in  Proposition \ref{square function Gmu}. 
    
Let us move to the cases $(p_1,p_2)\in \{(2,4),(4,4)\}$ (we remark that here we cannot use the symmetry over $f$ and $g$ as the bilinear multipliers $m_j^{\lambda}(\xi,\eta)$ are not symmetric with respect to $\xi$ and $\eta$). We first decompose the operator $\mathcal C_j^\lambda(f,g)(x)$ further dyadically in the following manner
\begin{align*}
|\mathcal C_j^{\lambda}(f,g)(x)|& \lesssim  \int_{\sqrt{2^{-2-j}}}^{\sqrt{2^{1-j}}}T_t^{\nu}g(x)B_{j,t}^{\mu}f(x)  t^{2\nu+1} dt + \sum\limits_{k=2+j}^{\infty}\int_{2^{-(k+1)/2}}^{2^{-k/2}} T_t^{\nu}g(x) B_{j,t}^{\mu}f(x)t^{2\nu+1} dt\\
& \le  \Big(\int_{\sqrt{2^{-2-j}}}^{\sqrt{2^{1-j}}}|T_t^{\nu}g(x) |^2 t^{2b}dt\Big)^{1/2} \Big(\int_{\sqrt{2^{-2-j}}}^{\sqrt{2^{1-j}}}|B_{j,t}^{\mu}f(x)|^2 t^{2a}dt\Big)^{1/2}\\
&\quad+\sum\limits_{k=2+j}^{\infty} \Big(\int_{2^{-(k+1)/2}}^{2^{-k/2}}|T_t^{\nu}g(x) |^2 t^{2b}dt\Big)^{1/2}\Big(\int_{2^{-(k+1)/2}}^{2^{-k/2}}|B_{j,t}^{\mu}f(x)|^2 t^{2a}dt\Big)^{1/2}\\
&=:  S^2g(x) S^1f(x)+\sum\limits_{k=2+j}^{\infty}S_k^2g(x)S_k^1f(x) .
\end{align*}
Then in particular
\begin{equation}{\label{dyadically estimate}}
\|\mathcal C_j^{\lambda}(f,g)\|_p \lesssim \|S^1f\|_{p_1}\|S^2g\|_{p_2}+ \sum\limits_{k=2+j}^{\infty}\|S_k^1f\|_{p_1} \|S_k^2g\|_{p_2}.    
\end{equation}
We remark that the above estimate \eqref{dyadically estimate} indeed holds for any $2 \le p_1,p_2 \le \infty$ and $\frac{1}{p}=\frac{1}{p_1}+\frac{1}{p_2}$; this will be useful in later Subsections \ref{412} and \ref{413}.

For the case $(p_1,p_2)=(2,4)$, we get that 
$$
\|\mathcal C_j^{\lambda}(f,g)\|_{\frac{4}{3}} \lesssim \|S^1f\|_{2}\|S^2g\|_{4}+ \sum\limits_{k=2+j}^{\infty}\|S_k^1f\|_{2} \|S_k^2g\|_{4}.
$$
Applying Proposition \ref{square function Gmu} part \eqref{Gmu estimate 1} and the estimate \eqref{L4Gnuk}, we have (the quantities $\varepsilon$ below may not be the same, but they are all positive) 
$$
\|S^1f\|_{2}\|S^2g\|_{4} \lesssim 2^{-\varepsilon j}j\|f\|_2 \|g\|_4 \lesssim 2^{-\varepsilon j}
\|f\|_2\|g\|_4.$$
For the second term, we use estimates \eqref{LpGmuk} with $p=2$ and \eqref{L4Gnuk},
$$
\sum\limits_{k=2+j}^{\infty}\|S_k^1f\|_{2} \|S_k^2g\|_{4} \lesssim  \sum\limits_{k=2+j}^{\infty} 2^{-k(\frac{a}{2}+\frac{1}{4})}2^{(1-\mu)j}k \|f\|_2 \|g\|_4 \lesssim 2^{-\varepsilon j}\|f\|_4\|g\|_4.$$

For the case $(p_1,p_2)=(4,4)$ we have that
$$
\|\mathcal C_j^{\lambda}(f,g)\|_{2} \lesssim \|S^1f\|_{4}\|S^2g\|_{4}+ \sum\limits_{k=2+j}^{\infty}\|S_k^1f\|_{4} \|S_k^2g\|_{4}.
$$
The first term in the above is handled using the estimate \eqref{L4Gnuk} and part \eqref{Gmu estimate 2} in Proposition \ref{square function Gmu}. For we have 
$$\|S^1f\|_{4}\|S^2g\|_{4} \lesssim 2^{-\varepsilon j}j\|f\|_4\|g\|_4 \lesssim 2^{-\varepsilon j}\|f\|_4\|g\|_4.$$
For the second term we use \eqref{L4Gnuk} and  \eqref{LpGmuk} with $p=4$, thus
$$
\sum\limits_{k=2+j}^{\infty}\|S_k^1f\|_{4} \|S_k^2g\|_{4} \lesssim  \sum\limits_{k=2+j}^{\infty} 2^{-k(\frac{a}{2}+\frac{1}{4})}2^{(1-\mu)j }k\|f\|_4\|g\|_4 \lesssim 2^{-\varepsilon j}\|f\|_4\|g\|_4.$$
This finishes the proof of part \eqref{main estimate 1}. 

\subsubsection{Proof of Proposition \ref{T_j estimate}, part \eqref{main estimate 4}} {\label{412}}

As we mentioned above, we cannot use the symmetry over $f$ and $g$, and hence we need to discuss the two cases separately: $p_{+}=p_1$, $p_{-}=p_2$, and $p_{+}=p_2$, $p_{-}=p_1$.

We first consider the case $2 \le p_{-}= p_2 \le 4 \le p_{+}=p_1 \le \infty$. In this case, we first interpolate between parts \eqref{Gmu estimate 3} (with $p = \infty$) and \eqref{Gmu estimate 2}  in  Proposition \ref{square function Gmu} and obtain that 
\begin{equation}{\label{4121}}
\|G_a^{\mu,j}f\|_{p_1} \lesssim 2^{-\varepsilon j}\|f\|_{p_1}    
\end{equation}
holds for $4 \le p_1 \le \infty$, $\mu > 1-\frac{2}{p_1}$ and $a>1/2$. 
 Further, we interpolate between \eqref{L4Gnuk} and part \eqref{Gnu estimate 1} of Proposition \ref{square function Gnu} with $j=k+1$, and obtain that, for $2 \le p_2 \le 4$, 
\begin{equation}{\label{4122}}
\Big\|\Big(\int^{2^{-k/2}}_{2^{-(k+1)/2}}|T^{\nu}_{t}g(x)|^2~t^{2b}dt\Big)^{1/2}\Big\|_{p_2} \lesssim 2^{-\varepsilon k}k^{2-\frac{4}{p_2}}\|g\|_{p_2},   
\end{equation}
where $\nu > -1/2$, $b > -1/2$ and $k \in \N$. For $\lambda > \frac{1}{2}-\frac{2}{p_{+}}=\frac{1}{2}-\frac{2}{p_1}$, we take $\widetilde{\lambda}=\lambda - (\frac{1}{2}-\frac{2}{p_1})$, $\mu =1 - \frac{2}{p_1}+\frac{\widetilde{\lambda}}{2}$, $\nu = b =-\frac{1}{2}+\frac{\widetilde{\lambda}}{2}$ and $a = \frac{1}{2}+\frac{\widetilde{\lambda}}{2}$. 
Continuing from estimate \eqref{dyadically estimate}, and applying estimates \eqref{4121} \eqref{4122} and \eqref{LpGmuk}, we get that 
\begin{align*}
\|\mathcal C_j^{\lambda}(f,g)\|_p & \lesssim \|S^1f\|_{p_1}\|S^2g\|_{p_2}+ \sum\limits_{k=2+j}^{\infty}\|S_k^1f\|_{p_1} \|S_k^2g\|_{p_2} \nonumber\\
& \lesssim 2^{-\varepsilon j} j^{2-\frac{4}{p_2}}\|f\|_{p_1}\|g\|_{p_2} +2^{(1-\mu)j}\sum\limits_{k=2+j}^{\infty}2^{-k(\frac{a}{2}+\frac{1}{4})}2^{-\varepsilon k} k^{2-\frac{4}{p_2}}\|f\|_{p_1}\|g\|_{p_2}\\
& \lesssim 2^{-\varepsilon j}\|f\|_{p_1}\|g\|_{p_2}. 
\end{align*}

Next, we consider the case $2 \le p_{-}=p_1 \le 4 \le p_{+} = p_2 \le \infty$. Interpolating between parts ({\ref{Gmu estimate 1}}) and ({\ref{Gmu estimate 2}}) in Proposition {\ref{square function Gmu}}, we obtain that for $2 \le p_1 \le 4$, $\mu > 1/2$ and $a >1/2$,
\begin{equation}{\label{4123}}
\|G_{a}^{\mu,j} f \|_{p_1} \lesssim 2^{-\varepsilon j}\|f\|_{p_1}.    
\end{equation}
We also interpolate between \eqref{L4Gnuk}  and part \eqref{Gnu estimate 2} of Proposition \ref{square function Gnu} with $j=k+1$ and $p = \infty$, and obtain that for $4 \le p_2 \le \infty$, 
\begin{equation}{\label{Gnu p>4}}
\Big\|\Big(\int^{2^{-k/2}}_{2^{-(k+1)/2}}|T^{\nu}_{t}f(x)|^2~t^{2b}dt\Big)^{1/2}\Big\|_{p_2} \lesssim   2^{-\varepsilon k}k^{\frac{4}{p_2}}\|f\|_{p_2}. \end{equation} 
holds for $\nu > -\frac{2}{p_2}$, $b > -1/2$ and $k \in \N$. For $\lambda > \frac{1}{2}-\frac{2}{p_{+}}=\frac{1}{2}-\frac{2}{p_2}$, we take $\widetilde {\lambda} = \lambda -(\frac{1}{2}-\frac{2}{p_2})$, $\nu = b = -\frac{2}{p_2}+ \frac{\widetilde{\lambda}}{2}$, $\mu=\frac{1}{2}+\frac{\widetilde{\lambda}}{2}$ and $a=1-\frac{2}{p_2}+\frac{\tilde{\lambda}}{2}$. In view of estimate \eqref{dyadically estimate}, and applying estimates \eqref{4123}, \eqref{Gnu p>4} and \eqref{LpGmuk},  we get that 
\begin{align*}
\|\mathcal C_j^{\lambda}(f,g)\|_p \lesssim 2^{-\varepsilon j}\|f\|_{p_1}\|g\|_{p_2}. 
\end{align*}
These prove part \eqref{main estimate 4} of Proposition~\ref{T_j estimate}. 

\subsubsection{Proof of Proposition \ref{T_j estimate}, part \eqref{main estimate 5}}{\label{413}}

In this case, for $\lambda > 1-\frac{2}{p}$, we take $\widetilde{\lambda}=\lambda - (1-\frac{2}{p})$, $\mu =1 - \frac{2}{p_1}+\frac{\widetilde{\lambda}}{2}$, $\nu = b = -\frac{2}{p_2}+\frac{\widetilde{\lambda}}{2}$ and $a = 1-\frac{2}{p_2}+\frac{\widetilde{\lambda}}{2}$. Starting again with estimate \eqref{dyadically estimate} and applying estimates \eqref{4121}, \eqref{Gnu p>4} and \eqref{LpGmuk},  we get that $
\|\mathcal C_j^{\lambda}(f,g)\|_p  \lesssim 2^{-\varepsilon j}\|f\|_{p_1}\|g\|_{p_2}$.
This completes the proof of Proposition~\ref{T_j estimate}, in the case $j\ge2$. 

\section{Proof of $L^p$ estimates for square functions, $p\neq 4$}\label{sec:pnot4}
In this section we will provide the proofs of the square function estimates, namely Proposition  \ref{square function Gnu} and Proposition \ref{square function Gmu}. The exposition will be systematic, splitting the proofs for each of the parts in the statements. At this point the only, but crucial, tool which will not be proven in this section are the $L^4$ estimates for the square functions. These are conceptually and technically involved and they will be proven in the subsequent Section \ref{sec:L4sf}.

\subsection{Proof of Proposition \ref{square function Gnu}}

\subsubsection{Proof of Proposition \ref{square function Gnu}, part \eqref{Gnu estimate 1}: $L^2$ estimate}

Let $\nu>-1/2$ and $b>-1/2$.
     Using Plancherel's theorem and Fubini's theorem, we have 
    \begin{align*}
        \Vert H^{\nu,j}_bg\Vert^2_{L^{2}}&=\int_{\R^2} \int^{\sqrt{2^{1-j}}}_{0} |T^{\nu}_{t}g(x)|^2 t^{2b} \,dt\, dx\\
        &= \int^{\sqrt{2^{1-j}}}_{0}  t^{2b} \int_{\R^2} \Big(1-\frac{\eta^2_1}{t^2\eta^2_2}\Big)_{+}^{2\nu} (\varphi(\eta_2))^2 |\widehat{g}(\eta)|^2 \,d\eta\,dt \\
        &=\int_{\R^2} (\varphi(\eta_2))^2 |\widehat{g}(\eta)|^2 \int^{\sqrt{2^{1-j}}}_{\frac{|\eta_1|}{\eta_2}} t^{2b} \Big(1-\frac{\eta^2_1}{t^2\eta^2_2}\Big)^{2\nu} \,dt\, d\eta\\
        &= \int_{\R^2} (\varphi(\eta_2))^2 |\widehat{g}(\eta)|^2  \Big(\frac{|\eta_1|}{\eta_2}\Big)^{2b+1} \int^{\sqrt{2^{1-j}}\frac{\eta_2}{|\eta_1|}}_{1} t^{2b-4\nu} (t^2-1)^{2\nu} \,dt\,  d\eta.
    \end{align*}
    We distinguish two cases. If $\sqrt{2^{1-j}}\frac{\eta_2}{|\eta_1|} \ge2$ we have, since $\nu>-\frac{1}{2}$,
    \begin{align*}
        \int^{\sqrt{2^{1-j}}\frac{\eta_2}{|\eta_1|}}_{1} t^{2b-4\nu} (t^2-1)^{2\nu} \,dt&=\int^2_{1} t^{2b-4\nu} (t^2-1)^{2\nu} \,dt+\int^{\sqrt{2^{1-j}}\frac{\eta_2} {|\eta_1|}}_2 t^{2b-4\nu} (t^2-1)^{2\nu} \,dt\\ &\lesssim 1+\int_{2}^{\sqrt{2^{1-j}}\frac{\eta_2}{|\eta_1|}} t^{2b-4\nu} (t+1)^{2\nu} (t-1)^{2\nu} \,dt\\
    &\lesssim \Big(\frac{\eta_2}{|\eta_1|}\Big)^{2b+1} 2^{-\frac{j(2b+1)}{2}},
    \end{align*}
    where we have used the fact that $t-1\geq \frac{t}{2}$ for $t\geq2$. On the other hand, if $1< \sqrt{2^{1-j}}\frac{\eta_2}{|\eta_1|} < 2$, we also have
    $$ \int^{\sqrt{2^{1-j}}\frac{\eta_2}{|\eta_1|}}_{1} t^{2b-4\nu} (t^2-1)^{2\nu} \,dt \lesssim 1 \lesssim \Big(\frac{\eta_2}{|\eta_1|}\Big)^{2b+1} 2^{-\frac{j(2b+1)}{2}}. $$    
 The above computations altogether imply the desired estimate.

    \subsubsection{Proof of Proposition \ref{square function Gnu}, part \eqref{Gnu estimate 2}: $L^p$ estimate, $2< p \le \infty$}
    
    Suppose $\nu >0$, $b > -1/2$.
Using Proposition \ref{Tnu} and Minkowski's inequality,
\begin{align*}
\|H^{\nu,j}_{b}f\|_p &= \Big\| \Big(\int^{\sqrt{2^{1-j}}}_{0} |T^{\nu}_{t}f(x)|^2 t^{2b} \, dt\Big)^{1/2} \Big\|_p  \le \Big\| \Big(\int^{\sqrt{2^{1-j}}}_{0} |\mathfrak{m}_t f(x)|^2 t^{2b} \, dt\Big)^{1/2}\Big\|_p  \\
& \le \Big(\int^{ \sqrt{2^{1-j}}}_{0} \|\mathfrak{m}_t f(x) \|_p^2 \,t^{2b} \, dt\Big)^{1/2} \lesssim 2^{-\frac{(2b+1)j}{4}}\|f\|_p, 
\end{align*}
and again we conclude the result in this case.

  \subsubsection{Proof of Proposition \ref{square function Gnu}, \eqref{L4Gnuk}: an improvement of $L^4$ estimate} 
  The proof of \eqref{L4Gnuk} is a delicate point and it will be shown in Section \ref{sec:L4sf}.

\subsection{Proof of Proposition \ref{square function Gmu}}

\subsubsection{Proof of Proposition \ref{square function Gmu}, part \eqref{Gmu estimate 1}: $L^2$ estimate}

 Let $\mu>1/2$ and $a>1/2$. 
 Using Plancherel's theorem and Fubini's theorem we get that
    \begin{align*}
        \Vert G^{\mu,j}_{a}f\Vert^2_{L^2}& = \int^{\sqrt{2^{1-j}}}_0  t^{2a}\int_{\R^2} \Big(\psi\Big(2^j\big(1-\frac{\xi^2_1}{\xi^2_2}\big)\Big)\Big)^2 (\varphi(\xi_2))^2 \Big( 1-\frac{\xi^2_1}{\xi^2_2}-t^2\Big)^{2\mu-2}_{+} |\widehat{f}(\xi)|^2 \,d\xi \,dt \\
        &= \int_{\R^2} \Big(\psi\Big(2^j\big(1-\frac{\xi^2_1}{\xi^2_2}\big)\Big)\Big)^2 (\varphi(\xi_2))^2 |\widehat{f}(\xi)|^2 \int^{\sqrt{1-\frac{\xi^2_1}{\xi^2_2}}}_{0} \Big( 1-\frac{\xi^2_1}{\xi^2_2}-t^2\Big)^{2\mu-2} t^{2a}\,dt\, d\xi.
    \end{align*}
    Now, 
    \begin{align*}
        \int^{\sqrt{1-\frac{\xi^2_1}{\xi^2_2}}}_{0} \Big( 1-\frac{\xi^2_1}{\xi^2_2}-t^2\Big)^{2\mu-2} t^{2a}\,dt&=\int^{\frac{1}{2}\sqrt{1-\frac{\xi^2_1}{\xi^2_2}}}_{0} \Big( 1-\frac{\xi^2_1}{\xi^2_2}-t^2\Big)^{2\mu-2} t^{2a}\,dt\\
        & \quad +\int^{\sqrt{1-\frac{\xi^2_1}{\xi^2_2}}}_{\frac{1}{2}\sqrt{1-\frac{\xi^2_1}{\xi^2_2}}} \Big( 1-\frac{\xi^2_1}{\xi^2_2}-t^2\Big)^{2\mu-2} t^{2a}\,dt.
    \end{align*}
    For $a>1/2$, the first component on the right hand side is dominated by $C_{\mu,a} (1-\frac{\xi^2_1}{\xi^2_2})^{2\mu+a-\frac{3}{2}}$.  On the other hand, by direct integration we get that the second component of the right hand side is dominated by $C_{\mu,a}(1-\frac{\xi^2_1}{\xi^2_2})^{2\mu+a-\frac{3}{2}}$, provided $\mu>1/2$. Now, using the support property of $\psi$, we get \[\Vert G^{\mu,j}_{a}f\Vert_{L^2} \lesssim 2^{-j(\mu+\frac{a}{2}-\frac{3}{4})}\Vert f\Vert_{L^2},
    \]
    which is the claimed estimate.

      \subsubsection{Proof of Proposition \ref{square function Gmu}, part \eqref{Gmu estimate 3}: $L^p$ estimate, $2 < p \le \infty$} Using Proposition \ref{Bmu} and Minkowski's integral inequality, we get
\begin{align*}
\|G^{\mu,j}_{a}f\|_p &= \Big\| \Big(\int^{\sqrt{2^{1-j}}}_{0} |B^{\mu}_{j,t}f(x)|^2  t^{2a} \, dt\Big)^{1/2} \Big\|_p   \\
&\le \Big\| \Big(\int^{\sqrt{2^{1-j}}}_{0} (1-t^2)^{2\mu-2}|\mathfrak{m}_{\sqrt{1-t^2}} \circ \mathfrak{m}_1 f(x)|^2 t^{2a} \, dt\Big)^{1/2}\Big\|_p  \\
& \le \Big(\int^{\sqrt{2^{1-j}}}_{0} \|\mathfrak{m}_{\sqrt{1-t^2}} \circ \mathfrak{m}_1 f(x) \|_p^2 \,t^{2a} \, dt\Big)^{1/2}  \\
& \lesssim 2^{-\frac{(2a+1)j}{4}}\|f\|_p, 
\end{align*}
and the proof is complete.

      \subsubsection{Proof of Proposition \ref{square function Gmu}, part \eqref{Gmu estimate 2}: an $L^4$ estimate} Analogously as for \eqref{L4Gnuk}, part \eqref{Gmu estimate 2} in Proposition \ref{square function Gmu} requires further attention and it will be shown in Section \ref{sec:L4sf}.

\subsubsection{Proof of \eqref{LpGmuk}: away from singularity part, $L^p$ estimate, $2\le p \le \infty$}{\label{Proof of LpGmuk}} Using Proposition \ref{operatorB} and Minkowski's integral inequality, we get
\begin{align*}
\Big\|\Big(\int^{2^{-k/2}}_{2^{-(k+1)/2}}|B^{\mu}_{j,t}f(x)|^2~t^{2a}dt\Big)^{1/2}\Big\|_p  &= 2^{(1-\mu)j}\Big\| \Big(\int^{2^{-k/2}}_{2^{-(k+1)/2}} |\mathfrak{m}_1f(x)|^2  t^{2a} \, dt\Big)^{1/2} \Big\|_p   \\
& \le 2^{(1-\mu)j}\Big(\int^{2^{-k/2}}_{2^{-(k+1)/2}}\| \mathfrak{m}_1 f(x) \|_p^2 \,t^{2a} \, dt\Big)^{1/2}  \\
& \lesssim 2^{-k(\frac{a}{2}+\frac{1}{4})}2^{(1-\mu)j}\|f\|_p. 
\end{align*}
This finishes the proof.

\section{$L^4$ estimates for square functions}
\label{sec:L4sf}
In this section we consider a geometric square function and establish its $L^4$ boundedness; such an estimate will be used to prove the local $L^4$ estimate for the square function $H_{b}^{\nu,j}$, namely \eqref{L4Gnuk} in Proposition~\ref{square function Gnu}, and the $L^4$ estimate for $G_a^{\mu,j}$, namely Proposition \ref{square function Gmu}, part \eqref{Gmu estimate 2}.

For $n\in \Z$, let $X_n$ denote the trapezoid in $\R^2$ with vertices at  $(2^{n-1},1/2)$, $(2^{n},1/2)$, $(2^{n+1},2)$ and $(2^{n+2},2)$. For a fixed $\ell\in \N$, we further decompose each $X_n$ into $2^{\ell}$ trapezoids $\{S_n^j\}_{j=1}^{2^{\ell}}$, where $S_n^j$ represents the trapezoid with its vertices at  $(2^{n-1}+(j-1)2^{n-1-{\ell}},1/2),(2^{n-1}+ j2^{n-1-{\ell}},1/2),(2^{n+1}+(j-1)2^{n+1-{\ell}},2)$ and $(2^{n+1}+ j2^{n+1-{\ell}},2)$. 
See Figure \ref{fig:Snj}.

\begin{center}
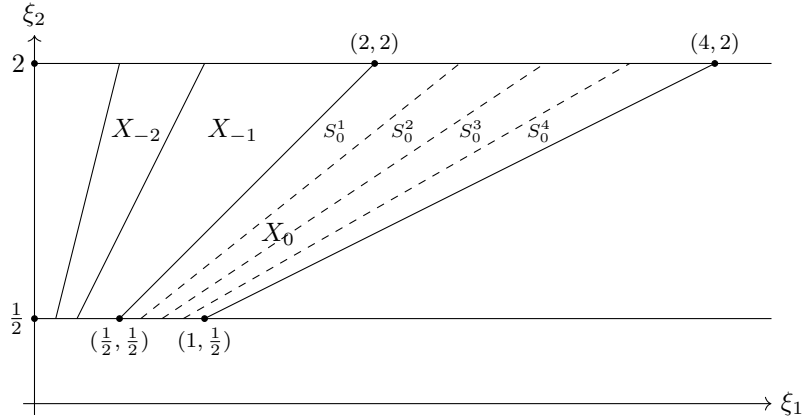
\begin{figure}
\begin{tikzpicture}[scale=0.75]
\draw[->] (-0.2,0) --(13,0) node[right] {$\xi_1$};
\draw[->] (0,-0.2) --(0,6.5) node[above] {$\xi_2$};
\draw (0,1.5) --(13,1.5);
\draw (0,6) --(13,6);
\filldraw (0,6) circle (.05);
\filldraw (0,6) node[left] {$2$};
\filldraw (0,1.5) circle (.05);
\filldraw (0,1.5) node[left] {$\frac{1}{2}$};
\draw (0.375,1.5) -- (1.5,6);
\draw (0.75,1.5) -- (3,6);
\draw (1.5,1.5) -- (6,6);
\filldraw (6,6) circle (.05);
\filldraw (6,6) node[above] {\footnotesize{$(2,2)$}};
\filldraw (1.5,1.5) circle (.05);
\filldraw (1.5,1.5) node[below] {\footnotesize{$(\frac{1}{2},\frac{1}{2})$}};
\draw (3,1.5) -- (12,6);
\draw [dashed] (1.875,1.5) -- (7.5,6);
\draw [dashed] (2.25,1.5) -- (9,6);
\draw [dashed] (2.625,1.5) -- (10.5,6);
\filldraw (12,6) circle (.05);
\filldraw (12,6) node[above] {\footnotesize{$(4,2)$}};
\filldraw (3,1.5) circle (.05);
\filldraw (3,1.5) node[below] {\footnotesize{$(1,\frac{1}{2})$}};
\node at (1.8,4.8) {{$X_{-2}$}};
\node at (3.5,4.8) {{$X_{-1}$}};
\node at (5.3,4.8) {\scriptsize{$S_{0}^{1}$}};
\node at (6.5,4.8) {\scriptsize{$S_{0}^{2}$}};
\node at (7.7,4.8) {\scriptsize{$S_{0}^{3}$}};
\node at (8.9,4.8) {\scriptsize{$S_{0}^{4}$}};
\node at (4.3,3.0) {{$X_{0}$}};
\end{tikzpicture} 
\caption{The trapezoids $X_{-2}, X_{-1}$ and $X_0$, and an illustration of $S_0^j$, $j=1,\ldots, 2^2$}
\label{fig:Snj}
\end{figure}
\end{center}

Then, 
$$
(0,\infty) \times [1/2,2]=\bigcup_{n\in \Z}\bigcup_{j=1}^{2^{\ell}}S_n^j.
$$
In what follows, with a slight abuse of notation, for each $A \subset \mathbb{R}^2$, we define $\widehat{Af}(\xi):=\chi_A(\xi)\widehat{f}(\xi)$. 

We prove an $L^4$ estimate for the corresponding square function. The approach is motivated by \cite{Cor1981, Carbery}.
\begin{proposition}\label{square function estimate}
There exists $\beta>0$ such that for every $\ell\in \N$, the inequality
$$
\Big\|\Big(\sum_{n\in \Z}\sum_{j=1}^{2^{\ell}}|S_n^jf|^2\Big)^{1/2}\Big\|_4 \lesssim {\ell}^{\beta}\|f\|_4
$$
holds, where the implicit constant is independent of $\ell$.
\end{proposition}
Let $0 < \alpha < 2\pi$. Consider the sector 
$$
\Omega_{\alpha}=\{\xi\in \R^2: 0 < \arg \xi <\alpha\}, 
$$
and divide it into $N$ equal angular sub-sectors $\{\Theta_j^{\alpha}\}_{j=1}^N := \{\xi\in \R^2: \frac{(j-1)\alpha}{N} < \arg \xi < \frac{j\alpha}{N}\}$. The proof of Proposition \ref{square function estimate} requires the following result for an angular square function. 
\begin{proposition}\label{directional square function estimate}
There exists $\beta_1 > 0$ such that for every $N>0$,
$$
\Big\|\Big(\sum_{j=1}^N|\Theta_j^{\alpha}f|^2\Big)^{1/2}\Big\|_4 \le C (\log N)^{\beta_1}\|f\|_4,
$$
where the implicit constant is independent of $N$ and $\alpha$.
\end{proposition}
Proposition \ref{directional square function estimate} is proved by employing the strategy used in \cite[Theorem 1]{Cor1982}. The key point is that a direct application of \cite[Theorem 1]{Cor1982} gives an estimate which depends on $\alpha$ and we need a uniform estimate in the angle $\alpha$. This requires a slight twist in the proof of  \cite[Theorem 1]{Cor1982}.   We will provide the main steps of the proof in the Appendix \ref{ap:proofds} for completeness of this paper. Applying Proposition {\ref{directional square function estimate}}, we can prove Proposition {\ref{square function estimate}}. This is shown in the next Subsection \ref{sub:propds}.

\subsection{Proof of Proposition \ref{square function estimate}}
\label{sub:propds}

We first decompose the square function into $\ell$ parts
$$
\Big(\sum_{n\in \Z}\sum_{j=1}^{2^{\ell}}|S_n^jf|^2\Big)^{1/2} = \Big(\sum\limits_{\gamma=0}^{\ell-1}\sum\limits_{n,j}|S_{n\ell+\gamma}^jf|^2\Big)^{1/2}\le \sum\limits_{\gamma=0}^{\ell-1}\Big(\sum\limits_{n,j}|S_{n\ell+\gamma}^jf|^2\Big)^{1/2}.$$
By symmetry, we only need to consider the case $\gamma =0$,
\begin{align}
{\label{estimate1}}
\notag \Big\|\Big(\sum\limits_{n,j}|S_{n\ell}^jf|^2\Big)^{1/2}\Big\|_4^4 &= \int_{\mathbb{R}^2} \Big(\sum\limits_{n,j}|S_{n\ell}^jf|^2\Big)^2 \, dx  \\
&= \sum\limits_{n}\int_{\mathbb{R}^2} \Big(\sum\limits_{j}|S_{n \ell}^jf|^2\Big)^2 \, dx + 2  \sum\limits_{n_1 > n_2}\sum\limits_{j,k}\int_{\mathbb{R}^2}|S_{n_1\ell}^jf|^2|S_{n_2\ell}^kf|^2 \, dx. 
\end{align}
For the first term in \eqref{estimate1} we use Proposition \ref{directional square function estimate} and the classical Littlewood--Paley theorem, so that there exists $\gamma >0 $ such that
\begin{align}
\label{estimate7} 
\sum\limits_{n \in \notag\mathbb{Z}}\int_{\mathbb{R}^2} \Big(\sum_{j=1}^{2^{\ell}}|S_{n\ell}^jf|^2\Big)^2 \, dx = \sum\limits_{n \in \mathbb{Z}}\int_{\mathbb{R}^2} \Big(\sum_{j=1}^{2^{\ell}}|S_{n\ell}^jX_nf|^2\Big)^2 \, dx & \lesssim 
\ell^{4\beta_1} \sum\limits_{n \in \mathbb{Z}}\int_{\mathbb{R}^2}|X_nf|^4 \, dx\\
&
\lesssim  \ell^{4\beta_1} \int_{\mathbb{R}^2}\Big(\sum\limits_{n \in \mathbb{Z}}|X_nf|^2 \Big)^2  \, dx
\lesssim   \ell^{4\beta_1} \|f\|_4^4,
\end{align}
where the implicit constant is independent of $\ell$.
For the second term in \eqref{estimate1}, we make a further decomposition on $S_{n_1\ell}^j$. For $n \in \mathbb{Z}$ and $0 < \alpha \le \frac{3}{2}2^\ell$, define
$$
S_{n}^{j,\alpha}:=\Big\{\xi \in S_{n}^j,\, \frac{1}{2}+ (\alpha-1) 2^{-\ell} < \xi_2 < \frac{1}{2}+ \alpha 2^{-\ell}\Big\},
$$
and write ${S_{n}^{\alpha}} := \cup_{j} S_{n}^{j,\alpha}$. It is not difficult to verify that there exists $C>0$ independent of $n_1, n_2,j,k,\ell$ such that 
$$\sum\limits_{\alpha}\chi_{S_{n_1\ell}^{j,\alpha}+S_{n_2\ell}^{k}} \le C.
$$
holds for every $n_1, n_2 \in \mathbb{Z}$, $1 \le j,k \le 2^{\ell}$ and $n_1 > n_2$. 

Thus,
\begin{align}
\notag \sum\limits_{n_1 > n_2}\sum\limits_{j,k}\int_{\mathbb{R}^2}|S_{n_1\ell}^jf|^2|S_{n_2\ell}^kf|^2 \, dx & = \sum\limits_{n_1 > n_2}\sum\limits_{j,k}\int_{\mathbb{R}^2}|\sum\limits_{\alpha}S_{n_1\ell}^{j,\alpha}f| ^2|S_{n_2\ell}^kf|^2 \, dx \nonumber\\
& = \sum\limits_{n_1 > n_2}\sum\limits_{j,k}\int_{\mathbb{R}^2}|\sum\limits_{\alpha} \widehat{S_{n_1\ell}^{j,\alpha}f} * \widehat{S_{n_2\ell}^kf}|^2 \, dx \nonumber \\
& \lesssim \sum\limits_{n_1 > n_2}\sum\limits_{j,k}\sum\limits_{\alpha}\int_{\mathbb{R}^2}| \widehat{S_{n_1\ell}^{j,\alpha}f} * \widehat{S_{n_2\ell}^kf}|^2 \, dx \nonumber \\
& \le \int_{\mathbb{R}^2}\sum\limits_{n_1,j,\alpha}|S_{n_1\ell}^{j,\alpha}f|^2\sum\limits_{n_2,k}|S_{n_2\ell}^kf|^2 \, dx \nonumber \\
& \le \Big\|\Big(\sum\limits_{n_1,j,\alpha}|S_{n_1 \ell}^{j,\alpha}f|^2\Big)^{1/2} \Big\|_4^2\Big\|\Big(\sum\limits_{n_2,k}|S_{n_2\ell}^kf|^2\Big)^{1/2}\Big\|_4^2. \label{estimate2}
\end{align}
Define $\mathfrak{m}^s(w):=(\mathfrak{m}(w^s))^{1/s}$, where $\mathfrak{m}$ is the usual strong maximal function, and choose a suitable positive function $w$ with $\|w\|_2=1$. We have, by applying twice Lemma \ref{summation lemma} with $s = \frac{3}{2}$,  
\begin{align}
\notag\Big\|\Big(\sum\limits_{n_1,j,\alpha}|S_{n_1 \ell}^{j,\alpha}f|^2\Big)^{1/2}\Big\|_4^2 & = \int_{\mathbb{R}^2} \sum\limits_{n_1}\sum\limits_{j,\alpha}|S_{n_1 \ell}^{j,\alpha}S_{n_1 \ell}^{\alpha}X_{n_1}f|^2 w \, dx   \\
& \lesssim \int_{\mathbb{R}^2} \sum\limits_{n_1}\sum\limits_{\alpha}|S_{n_1 \ell}^{\alpha} X_{n_1 \ell}f|^2 \mathfrak{m}^{3/2}(w) \, dx \nonumber \\
& \lesssim \int_{\mathbb{R}^2} \sum\limits_{n_1}|X_{n_1 \ell}f|^2 \mathfrak{m}^{3/2} \circ \mathfrak{m}^{3/2}(w) \, dx \nonumber \\
& \le \Big\|\Big(\sum\limits_{n_1}|X_{n_1 \ell}f|^2\Big)^{1/2}\Big\|_4^2\|\mathfrak{m}^{3/2} \circ\mathfrak{m}^{3/2}(w)\|_2 \nonumber \\
& \le \|f\|_4^2, \label{estimate3}
\end{align}
where we used the Littlewood--Paley theorem and the boundedness of the strong maximal function. We illustrate the strategy of the summation above in Figure \ref{summation over j and alpha}: in each of the summations in $j$ and $\alpha$, each trapezoid can be embedded into a one-parameter lattice of congruent strips parallel to the sides. 

\begin{center}
\begin{figure}
\begin{minipage}{0.25\linewidth}
\centering
\begin{tikzpicture}[scale=0.5]
\draw[->] (-0.2,0) --(6.5,0) node[right] {$\xi_1$};
\draw[->] (0,-0.2) --(0,6.5) node[above] {$\xi_2$};
\draw (0.75,1.5) --(1.5,1.5);
\draw (3,6) --(6,6);
\draw [dashed] (0,1.5) --(0.75,1.5);
\draw [dashed] (0,6) -- (3,6);
\filldraw (0,6) circle (.05);
\filldraw (0,6) node[left] {$2$};
\filldraw (0,1.5) circle (.05);
\filldraw (0,1.5) node[left] {$\frac{1}{2}$};
\draw (0.75,1.5) -- (3,6);
\draw (1.5,1.5) -- (6,6);
\draw [dashed] (0.9375,1.5) -- (3.75,6);
\draw [dashed] (1.125,1.5) -- (4.5,6);
\draw [dashed] (1.3125,1.5) -- (5.25,6);
\draw [dashed] (1.125,2.25) -- (2.25,2.25);
\draw [dashed] (1.5,3) -- (3,3);
\draw [dashed] (1.875,3.75) -- (3.75,3.75);
\draw [dashed] (2.25,4.5) -- (4.5,4.5);
\draw [dashed] (2.625,5.25) -- (5.25,5.25);
\node at (3.2,-1.3) {\scriptsize{Sets $\{S_{-1}^{j,\alpha}\}_{j,\alpha}$ with $1 \le j \le 4$}};
\node at (1.3,-2.0) {\scriptsize{and $1 \le \alpha \le 6$}};
\end{tikzpicture} 
\end{minipage}
$\xrightarrow{\text{\scriptsize{Sum over $j$}}}$
\begin{minipage}{0.25\linewidth}
\centering
\begin{tikzpicture}[scale=0.5]
\draw[->] (-0.2,0) --(6.5,0) node[right] {$\xi_1$};
\draw[->] (0,-0.2) --(0,6.5) node[above] {$\xi_2$};
\draw (0.75,1.5) --(1.5,1.5);
\draw (3,6) --(6,6);
\draw [dashed] (0,1.5) --(0.75,1.5);
\draw [dashed] (0,6) -- (3,6);
\filldraw (0,6) circle (.05);
\filldraw (0,6) node[left] {$2$};
\filldraw (0,1.5) circle (.05);
\filldraw (0,1.5) node[left] {$\frac{1}{2}$};
\draw (0.75,1.5) -- (3,6);
\draw (1.5,1.5) -- (6,6);
\draw [dashed] (1.125,2.25) -- (2.25,2.25);
\draw [dashed] (1.5,3) -- (3,3);
\draw [dashed] (1.875,3.75) -- (3.75,3.75);
\draw [dashed] (2.25,4.5) -- (4.5,4.5);
\draw [dashed] (2.625,5.25) -- (5.25,5.25);
\node at (3.2,-1.3) {\scriptsize{Sets $\{S_{-1}^{\alpha}\}_{\alpha}$ with $1 \le \alpha \le 6$}};
\node at (1.3,-2.0){};
\end{tikzpicture} 
\end{minipage}
$\xrightarrow{\text{\scriptsize{Sum over $\alpha$}}}$
\begin{minipage}{0.25\linewidth}
\centering
\begin{tikzpicture}[scale=0.5]
\draw[->] (-0.2,0) --(6.5,0) node[right] {$\xi_1$};
\draw[->] (0,-0.2) --(0,6.5) node[above] {$\xi_2$};
\draw (0.75,1.5) --(1.5,1.5);
\draw (3,6) --(6,6);
\draw [dashed] (0,1.5) --(0.75,1.5);
\draw [dashed] (0,6) -- (3,6);
\filldraw (0,6) circle (.05);
\filldraw (0,6) node[left] {$2$};
\filldraw (0,1.5) circle (.05);
\filldraw (0,1.5) node[left] {$\frac{1}{2}$};
\draw (0.75,1.5) -- (3,6);
\draw (1.5,1.5) -- (6,6);
\node at (3.2,-1.3) {\scriptsize{Set $X_{-1}$}};
\node at (1.3,-2.0){};
\end{tikzpicture} 

\end{minipage}
\caption{Summation over $j$ and $\alpha$}
\label{summation over j and alpha}
\end{figure}
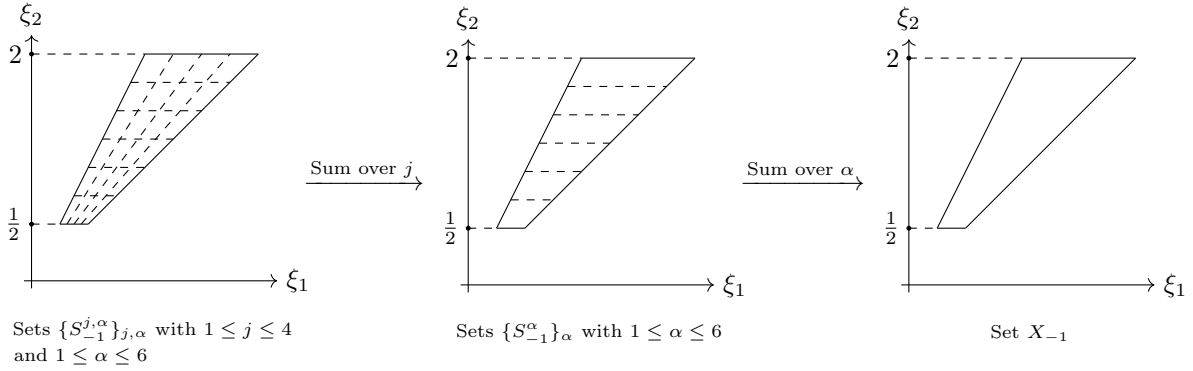
\end{center}

Combining estimates \eqref{estimate1}, \eqref{estimate7}, \eqref{estimate2} and \eqref{estimate3}, we have
\begin{equation}
\label{eq:boots}
\Big\|\Big(\sum\limits_{n,j}|S_{n\ell}^jf|^2\Big)^{1/2}\Big\|_4^4 \lesssim \ell^{4\beta_1} \|f\|_4^4 +\|f\|_4^2\Big\|\Big(\sum\limits_{n_2,k}|S_{n_2\ell}^kf|^2\Big)^{1/2}\Big\|_4^2.
\end{equation} 
A bootstrap argument yields that there exists $\beta > 0$, such that
$$
\Big\|\Big(\sum\limits_{n,j}|S_{n\ell}^jf|^2\Big)^{1/2}\Big\|_4 \le \ell^{\beta}\|f\|_4.
$$
Indeed, letting $A_{\ell}(f):=\Big\|\Big(\sum\limits_{n,j}|S_{n\ell}^jf|^2\Big)^{1/2}\Big\|_4$, inequality \eqref{eq:boots} says that
$$
A_{\ell}(f)^4\lesssim \ell^{4\beta_1} \|f\|_4^4 +\|f\|_4^2A_{\ell}(f)^2.
$$
Solving the inequality $A_{\ell}(f)^4-CA_{\ell}(f)^2\|f\|_4^2-C\ell^{4\beta_1}\|f\|_4^4\le 0$, we see that the positive root gives
$$
A_{\ell}(f)^2\lesssim  \|f\|_4^2+\ell^{2\beta_1}\|f\|_4^2.
$$
For large $\ell$, the dominant term is the $\ell^{2\beta_1}$ term, thus
$$
A_{\ell}(f)^2\lesssim \ell^{2\beta}\|f\|_4^2
$$
for some $\beta>0$.
This completes the proof of Proposition~\ref{square function estimate}. 
 
\subsection{Proof of Proposition \ref{square function Gnu}, \eqref{L4Gnuk}} 

In order to prove the estimate in \eqref{L4Gnuk}, we require $L^p$ estimates of the Kakeya maximal function defined by 
$$
\mathcal{K}_Nf(x):= \sup_{x \in R_N}\frac{1}{|R_N|}\int_{R_N}|f(y)|\,dy,
$$
where the supremum is taken over all rectangles $R_N$ with eccentricity $N$. The following estimate holds. 

\begin{lemma}\cite[Theorem 5.3.5]{Grafakosmodern}
There exists $C>0$ independent of $N$ such that
$$
\|\mathcal{K}_Nf\|_{2} \le C \log N\|f\|_2.
$$   
\end{lemma}
Now, we define a variant of Kakeya maximal function, let
\begin{equation}{\label{maximal function definition}}
\mathcal{K}_{a,b}f(x) := \sup_{x \in R_{a,b}}\frac{1}{|R_{a,b}|}\int_{R_{a,b}}|f(y)| \,dy,    
\end{equation}
where the supremum is over all rectangles $R_{a,b}$ with eccentricity between $a$ and $b$. Then, for each $\ell \in \N$
\begin{equation}{\label{estimate}}
\|\mathcal{K}_{1,2^{\ell}}f\|_2 \le \sum\limits_{k=0}^{\ell-1}\|\mathcal{K}_{2^k,2^{k+1}}f\|_2 
\lesssim \sum\limits_{k=0}^{\ell-1}k\|f\|_2\lesssim \ell^2\|f\|_2.     
\end{equation}
We proceed with the proof of the estimate \eqref{L4Gnuk}.

Recall the decomposition of the operator $T_t^{\nu}$ as in \eqref{partitionofunity},
$$|{T}^{\nu}_tg(x)|\leq |T^{\nu,\psi_1}_{t}g(x)|+\sum_{\ell \ge 2}  2^{-\ell\nu}|T^{\nu,\psi}_{2^{-\ell},t}g(x)|,\quad x\in \R^2,$$
where the operators $T_{t}^{\nu,\psi_1}$ and $T^{\nu,\psi}_{2^{-\ell},t}$ are defined, respectively, in \eqref{Tuno} and \eqref{def:Tnudelta}. For the non-singular part $T_{t}^{\nu,\psi_1}$, it was shown in the proof of Proposition \ref{Tnu} that $\left|T_{t}^{\nu,\psi_1}g(x)\right| \lesssim \mathfrak{m}_{t}g(x)$, and as a result, using Minkowski's inequality, we have
$$
\Big\|\Big(\int^{2^{-k/2}}_{2^{-(k+1)/2}}|T_{t}^{\nu,\psi_1}g(x)|^2 t^{2b} \, d t\Big)^{1/2}
\Big\|_{4} \lesssim \|g\|_4.
$$

For the other part $T^{\nu,\psi}_{2^{-\ell},t}$, in order to simplify the notation, we denote 
$$
\widehat{T^{\nu,\psi}_{2^{-\ell},t}g}(\xi) = \Phi_{\ell}\Big(\frac{\xi_1}{t},{\xi_2}\Big)\widehat{g}(\xi), \quad \xi=(\xi_1,\xi_2),
$$
where $\Phi_{\ell}(\xi):=\psi\big(2^{\ell}\big(1-\frac{\xi^2_1}{\xi^2_2}\big)\big) \big(2^{\ell}(1-\frac{\xi^2_1}{\xi^2_2})\big)^{\nu}\varphi(\xi_2)$. Our goal is to prove that there exists $\alpha > 0$ such that, for every $\ell \ge 2$,
$$
\Big\Vert\Big(\int_{2^{-\frac{k+1}{2}}}^{2^{-\frac{k}{2}}} |T^{\nu,\psi}_{2^{-\ell},t}g(x)|^2 t^{2b} \,{d t}\Big)^{1/2}\Big\Vert_{4} \lesssim 2^{-\ell/2}\ell^{\alpha}k\|g\Vert_4.
$$
Indeed, assuming the above and adding in $\ell$ we obtain the result, under the restriction $\nu>-1/2$.

Given $\ell\ge2$ recall that, for $n\in \Z$ and $1\le j\le 2^{\ell}$, $S_n^j$ are the trapezoids introduced at the beginning of the current section. For $0<t<1$, let $\mathcal{A}_t:=\{S_n^j, \,\,\operatorname{supp} \Phi_{\ell}(\frac{\cdot}{t},\cdot) \cap S_n^j \neq \emptyset \}$. It is easy to verify that for every $t \in \mathbb{R}$, $\mathcal{A}_t$ contains at most $C$ trapezoids, where $C$ is a constant independent of $t$ and $\ell$. Then,
\begin{align}
&\Big\|\Big(\int_{2^{-\frac{k+1}{2}}}^{2^{-\frac{k}{2}}}|T^{\nu,\psi}_{2^{-\ell},t}g(x)|^2 t^{2b} \, {d t}\Big)^{1/2}\Big\|_{4}^4 = \int_{\mathbb{R}^2}\int_{2^{-\frac{k+1}{2}}}^{2^{-\frac{k}{2}}} |T^{\nu,\psi}_{2^{-\ell},t}g(x)|^2 t^{2b} \, {d t}\int_{2^{-\frac{k+1}{2}}}^{2^{-\frac{k}{2}}} |T^{\nu,\psi}_{2^{-\ell},u}g(x)|^2 u^{2b} \, {d u}\, d x \nonumber \\ 
&\qquad = \int_{2^{-\frac{k+1}{2}}}^{2^{-\frac{k}{2}}}\int_{2^{-\frac{k+1}{2}}}^{2^{-\frac{k}{2}}}\int_{\mathbb{R}^2}|T^{\nu,\psi}_{2^{-\ell},t}g(x)|^2 |T^{\nu,\psi}_{2^{-\ell},u}g(x)|^2 \,d x\,t^{2b} \, {d t}\, u^{2b} \, {d u} \nonumber \\
& \qquad =  \int_{2^{-\frac{k+1}{2}}}^{2^{-\frac{k}{2}}}\int_{2^{-\frac{k+1}{2}}}^{2^{-\frac{k}{2}}}\int_{\mathbb{R}^2} \big|\big(\Phi_{\ell}\big(\frac{\cdot}{t},\cdot\big)\widehat{g}(\cdot,\cdot)\big) \ast \big(\Phi_{\ell}\big(\frac{\cdot}{u},{\cdot}\big)\widehat{g}(\cdot,\cdot)\big)\big|^2\, d \xi\, t^{2b} \, {d t} \,u^{2b} \,{d u} \nonumber \\
&\qquad  =  \int_{2^{-\frac{k+1}{2}}}^{2^{-\frac{k}{2}}}\int_{2^{-\frac{k+1}{2}}}^{2^{-\frac{k}{2}}}\int_{\mathbb{R}^2} \Big|\sum\limits_{n,n^{\prime},j,j^{\prime}}\big(\Phi_{\ell}\big(\frac{\cdot}{t},\cdot\big)\widehat{S_n^jg}(\cdot,\cdot)\big) \ast \big(\Phi_{\ell}(\frac{\cdot}{u},{\cdot}\big)\widehat{S_{n^{\prime}}^{j^{\prime}}g}(\cdot,\cdot)\big)\Big|^2\, d \xi \,t^{2b} \, {d t} \,u^{2b} \,{d u} \nonumber \\
& \qquad \lesssim  \sum\limits_{n,n^{\prime},j,j^{\prime}} \int_{2^{-\frac{k+1}{2}}}^{2^{-\frac{k}{2}}}\int_{2^{-\frac{k+1}{2}}}^{2^{-\frac{k}{2}}}\int_{\mathbb{R}^2} \big|\big(\Phi_{\ell}\big(\frac{\cdot}{t},\cdot)\widehat{S_n^jg}(\cdot,\cdot)\big) \ast \big(\Phi_{\ell}\big(\frac{\cdot}{u},{\cdot}\big)\widehat{S_{n^{\prime}}^{j^{\prime}}g}(\cdot,\cdot)\big)\big|^2\, d \xi \,t^{2b} \, {d t} \,u^{2b} \, {d u} \nonumber \\
& \qquad = \sum\limits_{n,n^{\prime},j,j^{\prime}} \int_{2^{-\frac{k+1}{2}}}^{2^{-\frac{k}{2}}}\int_{2^{-\frac{k+1}{2}}}^{2^{-\frac{k}{2}}}\int_{\mathbb{R}^2} |T^{\nu,\psi}_{2^{-\ell},t}S_n^jg(x)|^2|T^{\nu,\psi}_{2^{-\ell},u}S_{n^{\prime}}^{j^{\prime}}g(x)|^2\, d x\, t^{2b} \, dt \, u^{2b} \, du \nonumber \\
&\qquad  = \Big \|\Big(\int_{2^{-\frac{k+1}{2}}}^{2^{-\frac{k}{2}}} \sum\limits_{n,j}|T^{\nu,\psi}_{2^{-\ell},t}S_n^jg(x)|^2 t^{2b} \, {d t}\Big)^{1/2}\Big\|_{4}^4, \label{equal new square function}
\end{align}
where, in view of the support of $\operatorname{supp} \Phi_{\ell}(\frac{\cdot}{t},\cdot) \cap S_n^j$, the sum in $n, n^{\prime}$ runs from $-\infty$ to $0$, and the sum in $j, j^{\prime}$ runs from $0$ to $2^{\ell}$. 

For each $j, n$ fixed, let $\mathcal{B}_{n}^{j} =\{t \in [0,\infty): \exists\,x \in \mathbb{R}^2 \,\text{ such that }\, T^{\nu,\psi}_{2^{-\ell}, t}S_n^jg(x)\neq 0 \}$. Then, since $n\le 0$,
$$
\int_{\mathcal{B}_n^j} t^{2b} \, d{t} \le \int_{\{t:\operatorname{supp} \Phi_{\ell}(\frac{\cdot}{t},\cdot) \cap S_n^j \neq \phi \}} t^{2b} \, dt \lesssim 2^{n-\ell}{2^{2nb}} \le 2^{-\ell}.
$$
Invoking Corollary~\ref{eccentricity result}, we have that there exists a sequence $\{C_i\}_{i=1}^{\infty}$ with $\|\{C_i\}\|_{\ell^1}=1$ such that
$$
|T^{\nu,\psi}_{2^{-\ell},t}g(x)| \lesssim \sum_{i}C_i\frac{\chi_{A_i^t}}{|A_i^t|} \ast |g|(x),
$$
where $A_i^t$ are rectangles with eccentricity between $1$ and $2^{\ell}t^{-1}$.

Note that this reduces our task to showing that there exists $\alpha >0$, such that 
$$
\Big\|\Big(\int_{2^{-\frac{k+1}{2}}}^{2^{-\frac{k}{2}}} \sum\limits_{n,j} \Big|\chi_{\mathcal{B}_n^j}(t)\Big(\frac{\chi_{A_i^t}}{|A_i^t|}*|S_n^jg|\Big)(x)\Big|^2 t^{2b} \, {d t}\Big)^{1/2}\Big\|_{4} \lesssim 2^{-\ell/2} \ell^{\alpha}k\|g\|_4
$$
holds uniformly for every $i$. 
Indeed, for every $h \in L^2({\mathbb{R}^n})$ with $\|h\|_2=1$, we have 
\begin{align}
&\int_{\mathbb{R}^2} \int_{2^{-\frac{k+1}{2}}}^{2^{-\frac{k}{2}}}\sum\limits_{n,j} \Big|\chi_{\mathcal{B}_n^j}(t)\Big(\frac{\chi_{A_i^t}}{|A_i^t|}*|S_n^jg|\Big)(x)\Big|^2 t^{2b} \, dt \, h(x)\, dx \nonumber\\
&\quad \le  \sum\limits_{n,j} \int_{2^{-\frac{k+1}{2}}}^{2^{-\frac{k}{2}}}\chi_{\mathcal{B}_n^j}(t)\int_{\mathbb{R}^2}  \Big(\frac{\chi_{A_i^t}}{|A_i^t|}*|S_n^jg|^2\Big)(x) \cdot h(x)\, dx\, t^{2b} \,d t \nonumber \\
&\quad=  \sum\limits_{n,j} \int_{2^{-\frac{k+1}{2}}}^{2^{-\frac{k}{2}}}\chi_{\mathcal{B}_n^j}(t)\int_{\mathbb{R}^2} |S_n^jg(x)|^2 \Big(\frac{\chi_{\widetilde{A}_i^t}}{|A_i^t|}* h\Big)(x)\, dx\, t^{2b} \, dt \nonumber \\
&\quad\le  \sum\limits_{n,j} \int_{\mathbb{R}} \chi_{\mathcal{B}_n^j}(t) \int_{\mathbb{R}^2} |S_n^jg(x)|^2\mathcal{K}_{1,2^{\ell+\frac{k+1}{2}}}h(x) \,dx\, t^{2b} \, dt \nonumber \\
&\quad\le 2^{-\ell} \int_{\mathbb{R}^2} \sum\limits_{n,j}|S_n^jg(x)|^2 \mathcal{K}_{1,2^{\ell+\frac{k+1}{2}}}h(x)\, dx \nonumber \\
&\quad\le  2^{-\ell} \Big\|\Big(\sum\limits_{n,j}|S_n^jg(x)|^2\Big)^{1/2}\Big\|_{4}^2 \|\mathcal{K}_{1,2^{\ell+\frac{k+1}{2}}}h\|_2  \nonumber \\
&\notag \quad\lesssim  2^{-\ell} \ell^{2 \beta}\Big(\ell+\frac{k+1}{2}\Big)^2\|g\|_4^2 \\
&\quad \lesssim  2^{-\ell} \ell^{2\beta+2}k^2\|g\|_4^2, 
\end{align}
where $\widetilde{A}_i^t:=-A_i^t$ and we used Proposition \ref{square function estimate} and \eqref{estimate} in the second to last step.

\subsection{Proof of Proposition \ref{square function Gmu}, part \eqref{Gmu estimate 2}}

We recall the operator ${B}^{\mu}_{j,t}$ given in \eqref{eq:Bjt}.
Now, for $t \in [0,\sqrt{2^{1-j}}]$, $j\geq2$, and $ \mu > \frac{1}{2}$, we have
\begin{align*}
{B}^{\mu}_{j,t}f(x) & = (1-t^2)^{\mu-1}\int_{\mathbb{R} \times [1/2,2]}\Big(1-\frac{\xi_1^2}{(1-t^2)\xi^2_2}\Big)^{\mu-1}_{+}\varphi(\xi_2)\psi\Big(2^{j}\big(1-\frac{\xi_1^2}{\xi^2_2}\big)\Big)\widehat{f}(\xi) e^{2\pi ix\cdot\xi} d\xi   \\ 
& = (1-t^2)^{\mu-1}{T}_{\sqrt{1-t^2}}^{\mu-1} \, f_{1,j}(x),  
\end{align*}
where $T^{\mu-1}_{\sqrt{1-t^2}}$ is the operator defined in \eqref{eq:Toperator}, and $f_{1,j}(x) = K_{1,j} * f(x)$ is defined in ({\ref{fRj}}). Then, 
\begin{align*}
\Big\|\Big(\int^{\sqrt{2^{1-j}}}_{0}|{B}^{\mu}_{j,t}f(x)|^2 t^{2a} \, dt\Big)^{1/2}\Big\|_4 & = \Big\|\Big(\int^{\sqrt{2^{1-j}}}_{0}|{T}_{\sqrt{1-t^2}}^{\mu-1} \, f_{1,j}(x)|^2 (1-t^2)^{2\mu-2} t^{2a} \, dt\Big)^{1/2}\Big\|_4  \\
& = \Big\|\Big(\int_{\sqrt{1-2^{1-j}}}^{1}|{T}_{s}^{\mu-1} \, f_{1,j}(x)|^2 (1-s^2)^{a-1/2}\, s^{4\mu-3} \, ds\Big)^{1/2}\Big\|_4. 
\end{align*}
Now, we apply again the decomposition \eqref{partitionofunity} to the operator ${T}_{s}^{\mu-1}$. Then, for the non-singular part $T_{s}^{\mu-1,\psi_1} $, using estimates \eqref{estimate of inner part psi_1} and \eqref{fj}, we obtain
\begin{align*}
&\Big\|\Big(\int_{\sqrt{1-2^{1-j}}}^{1}|T_{s}^{\mu-1,\psi_1} \, f_{1,j}(x)|^2 (1-s^2)^{a-1/2}\, s^{4\mu-3} \, ds\Big)^{1/2}\Big\|_4 \\
&\qquad 
\lesssim \Big \|\Big(\int_{\sqrt{1-2^{1-j}}}^{1}|\mathfrak{m}_{s} \circ\mathfrak{m}_1f(x)|^2 (1-s^2)^{a-1/2}\, s^{4 \mu -3} \, ds\Big)^{1/2}\Big\|_4 \\
&\qquad \lesssim 2^{-j(\frac{a}{2}+\frac{1}{4})}\|f\|_4.
\end{align*}

For the singular part $T^{\mu-1,\psi}_{2^{-\ell},s}$ our goal is to show that there exists $\alpha, \varepsilon> 0$ such that
$$
\Big\|\Big(\int_{\sqrt{1-2^{1-j}}}^{1}|T^{\mu-1,\psi}_{2^{-\ell},s} \, f_{1,j}(x)|^2 (1-s^2)^{a-1/2}\, s^{4\mu-3} \, ds\Big)^{1/2}\Big\|_4 \lesssim 2^{-\varepsilon j}2^{-\ell/2}\ell ^{\alpha}\|f\|_4.
$$
Once we obtain the above inequality, in view of \eqref{partitionofunity}, we conclude the desired estimate under the assumption that $\mu>1/2$.
Without loss of generality, we assume 
$$
\operatorname{supp}(\widehat{f}) \subset \Big\{\xi \in \mathbb{R}^2: \frac{1}{4} < \xi_1/\xi_2 < 1,\, \frac{1}{2}< \xi_2 < 2\Big\}.
$$
Now we decompose the above region into $2^{\ell}$ trapezoids $\{S_k\}_{k=1}^{2^{\ell}}$, where $S_k$ represents the trapezoid with endpoints 
$(\frac{1}{8} + (k-1) \frac{3}{8}2^{-\ell}, \frac{1}{2})$, $(\frac{1}{8} + k \frac{3}{8}2^{-\ell}, \frac{1}{2})$, $(\frac{1}{2} + (k-1) \frac{3}{2}2^{-\ell}, 2)$ and $(\frac{1}{2} + k \frac{3}{2}2^{-\ell},2)$. 
Similarly as in the previous subsection, it suffices to prove that there exists $\alpha, \varepsilon> 0$ such that,
\begin{equation}
\label{eq:sub63}
\Big\|\Big(\int_{\sqrt{1-2^{1-j}}}^{1} \sum\limits_{k=1}^{2^{\ell}}|T^{\mu-1,\psi}_{2^{-\ell},s}\, S_kf_{1,j}(x)|^2 (1-s^2)^{a-1/2}\, s^{4\mu-3} \, ds\Big)^{1/2}\Big\|_4 \lesssim 2^{-\varepsilon j}2^{-\ell/2}\ell^{\alpha} \|f\|_4.
\end{equation}
On the one hand, analogously as before, for each $k,j$ fixed, let 
$$
\mathcal{B}_{k}^j =\{s \in (\sqrt{1-2^{1-j}},1): \exists \,x \in \mathbb{R}^2 \text{ such that } T^{\mu-1,\psi}_{2^{-\ell},s}S_kf_{1,j}(x)\neq 0 \}.
$$
Thus, in order to prove \eqref{eq:sub63}, in view of Corollary \ref{eccentricity result},  we only need to show that, for every $i\ge 1$, 
$$
\Big\|\Big(\int_{\sqrt{1-2^{1-j}}}^{1}\sum\limits_{k} \Big|\chi_{\mathcal{B}_k^j}(s)\Big(\frac{\chi_{A_i^s}}{|A_i^s|}*|S_kf_{1,j}|\Big)(x)\Big|^2(1-s^2)^{a-1/2}\, s^{4\mu-3} \, ds\Big)^{1/2}\Big\|_{4} \lesssim 2^{-\varepsilon j}2^{-\ell/2}\ell^{\alpha} \|f\|_4,
$$
where $A_i^s$ are rectangles with eccentricity between $1$ and $2^{\ell +1 }$. Observe also that 
$$
\int_{\mathcal{B}_k^j} (1-s^2)^{a-1/2}\, s^{4\mu-3} \, ds \lesssim 2^{-j(a-1/2)}{|\{s:\operatorname{supp} \Phi_{\ell}(\tfrac{\cdot}{s},\cdot) \cap S_k \neq \emptyset \}}| \lesssim 2^{-j(a-1/2)}2^{-\ell}.
$$
Now, for every $h \in L^2({\mathbb{R}^n})$ with $\|h\|_2=1$,
\begin{align*}
&   \int_{\mathbb{R}^2} \int_{\sqrt{1-2^{1-j}}}^{1}\sum\limits_{k} \Big|\chi_{\mathcal{B}_k^j}(s)\Big(\frac{\chi_{A_i^s}}{|A_i^s|}*|S_kf_{1,j}|\Big)(x)\Big|^2 (1-s^2)^{a-1/2}\, s^{4\mu-3} \, ds  \,h(x) dx \\
& \quad\le  \sum\limits_{k} \int_{\sqrt{1-2^{1-j}}}^{1} \chi_{\mathcal{B}_k^j}(s)\int_{\mathbb{R}^2}  \Big(\frac{\chi_{A_i^s}}{|A_i^s|}*|S_kf_{1,j}|^2\Big)(x) \cdot h(x) dx (1-s^2)^{a-1/2}\, s^{4\mu-3} \, ds  \\
& \quad=  \sum\limits_{k} \int_{\sqrt{1-2^{1-j}}}^{1}  \chi_{\mathcal{B}_k^j}(s)\int_{\mathbb{R}^2} |S_kf_{1,j}(x)|^2 \Big(\frac{\chi_{\widetilde{A}_i^s}}{|A_i^s|}* h\Big)(x)\, dx\, (1-s^2)^{a-1/2}\, s^{4\mu-3} \, ds  \\
& \quad\le  \sum\limits_{k} \int_{\mathbb{R}} \chi_{\mathcal{B}_k^j}(s) \int_{\mathbb{R}^2} |S_kf_{1,j}(x)|^2 \mathcal{K}_{1,2^{\ell+1}}h(x) \,dx\, (1-s^2)^{a-1/2}\, s^{4\mu-3} \, ds \\
& \quad\le  2^{-j(a-1/2)}2^{-\ell} \int_{\mathbb{R}^2} \sum\limits_{k}|S_kf_{1,j}(x)|^2 \mathcal{K}_{1,2^{\ell+1}}h(x) \,dx \\
& \quad\le  2^{-j(a-1/2)}2^{-\ell} \Big\|\big(\sum\limits_{k}|S_kf_{1,j}(x)|^2\big)^{1/2}\Big\|_{4}^2 \|\mathcal{K}_{1,2^{\ell+1}}h(x)\|_2   \\
& \quad\lesssim  2^{-j(a-1/2)}2^{-\ell}l^{2\beta + 2}\|f_{1,j}\|_4^2 \le 2^{-j(a-1/2)}2^{-\ell}l^{2\beta + 2}\|f\|_4^2 , 
\end{align*}
where $\widetilde{A}_i^s = -A_i^s$ and we used Proposition \ref{square function estimate}, estimate \eqref{estimate} and \eqref{fj}.

\section{Proof of Proposition {\ref{T_j estimate}} for the term $\mathcal C_1^{\lambda}(f,g)$}\label{proofofj=1}\label{j=1case}

In view of \eqref{eq:Cljoriginal} and \eqref{eq:ml1}, we can write 
\begin{align*}
    \mathcal C^{\lambda}_1(f,g)(x)=\int_{\mathbb{R}^2\times \mathbb{R}^2} \psi_1\Big(\frac{\xi^2_1}{\xi^2_2}\Big) \Big(1-\frac{\xi^2_1}{\xi^2_2}-\frac{\eta^2_1}{\eta^2_2}\Big)^{\lambda}_{+} \varphi(\xi_2)\varphi(\eta_2) \widehat{f}(\xi)\widehat{g}(\eta) e^{2\pi i x\cdot(\xi+\eta)} d\xi\, d\eta. 
\end{align*}
We decompose the multiplier further with respect to $\eta$ as follows.  
\begin{align*}
    \psi_1\Big(\frac{\xi^2_1}{\xi^2_2}\Big)  \Big(1-\frac{\xi^2_1}{\xi^2_2}-\frac{\eta^2_1}{\eta^2_2}\Big)^{\lambda}_{+}\varphi(\xi_2)\varphi(\eta_2)
    &=\sum\limits_{k \ge 2} m^{\lambda}_{1,k}(\xi,\eta)+m_{1,1}^{\lambda}(\xi,\eta),
\end{align*}
where
$$
 m^{\lambda}_{1,k}(\xi,\eta):=\psi\Big(2^k\big(1-\frac{\eta^2_1}{\eta^2_2}\big)\Big) \psi_1\Big(\frac{\xi^2_1}{\xi^2_2}\Big)  \Big(1-\frac{\xi^2_1}{\xi^2_2}-\frac{\eta^2_1}{\eta^2_2}\Big)^{\lambda}_{+}\varphi(\xi_2)\varphi(\eta_2)
$$
and 
$$
m_{1,1}^{\lambda}(\xi,\eta):= \psi_1\Big(\frac{\eta^2_1}{\eta^2_2}\Big) \psi_1\Big(\frac{\xi^2_1}{\xi^2_2}\Big) \Big(1-\frac{\xi^2_1}{\xi^2_2}-\frac{\eta^2_1}{\eta^2_2}\Big)^{\lambda}_{+}\varphi(\xi_2)\varphi(\eta_2).
$$
Additionally, we decompose 
\begin{align*}
m_{1,1}^{\lambda}(\xi,\eta) &= \varphi(\xi_2)\varphi(\eta_2)\bigg[\psi_{11}\Big(\frac{\eta^2_1}{\eta^2_2}\Big) \psi_1\Big(\frac{\xi^2_1}{\xi^2_2}\Big) \Big(1-\frac{\xi^2_1}{\xi^2_2}-\frac{\eta^2_1}{\eta^2_2}\Big)^{\lambda}_{+} +\psi_{12}\Big(\frac{\eta^2_1}{\eta^2_2}\Big) \psi_1\Big(\frac{\xi^2_1}{\xi^2_2}\Big) \Big(1-\frac{\xi^2_1}{\xi^2_2}-\frac{\eta^2_1}{\eta^2_2}\Big)^{\lambda}_{+}\Big]\\
&=: m_{1,11}^{\lambda}(\xi,\eta) + m_{1,12}^{\lambda}(\xi,\eta),
\end{align*}
where $\psi_{11}$ and $\psi_{12}$ are smooth functions and satisfying
$$
\operatorname{supp}(\psi_{11})\subset [-\tfrac{1}{8},\tfrac{1}{8}]\quad \text{ and } \quad \operatorname{supp}(\psi_{12})\subset [\tfrac{1}{16},\tfrac{3}{4}].
$$
Since the multiplier $ m_{1,11}^
{\lambda}(\xi,\eta)$ is a smooth function,  the operator corresponding to this part is bounded from $L^{p_1}\times L^{p_2}$ to $L^p$ for all $1 \le p,p_1,p_2\leq\infty$ and $\frac{1}{p}=\frac{1}{p_1}+\frac{1}{p_2}$.
Let $T_{1,k}^{\lambda}$ and $T_{1,12}^{\lambda}$ denote the bilinear multiplier operators corresponding to the multipliers $m^{\lambda}_{1,k}(\xi,\eta)$ and $m^{\lambda}_{1,12}(\xi,\eta)$ respectively. 
Therefore, the case $j=1$ in Proposition \ref{T_j estimate} is reduced to proving the following propositions.

\begin{proposition}
\label{prop1inj=1}
Let $k \ge 2$, suppose $2 \le p_1,p_2 \le \infty$ with $\frac{1}{p}=\frac{1}{p_1}+\frac{1}{p_2}$. Then for $p_1, p_2,  \lambda$ satisfying the assumptions of three parts \eqref{main estimate 1}, \eqref{main estimate 4} and \eqref{main estimate 5} in Proposition \ref{T_j estimate},
$$
\|T_{1,k}^{\lambda}(f,g)\|_p \lesssim 2^{-\varepsilon k}\|f\|_{p_1}\|g\|_{p_2}
$$
holds for $f \in L^{p_1}(\R^2)$, $g \in L^{p_2}(\R^2)$.
\end{proposition}

\begin{proposition}{\label{T1j operator}}
	Suppose $2 \le p_1,p_2 \le \infty$ with $\frac{1}{p}=\frac{1}{p_1}+\frac{1}{p_2}$. Then for $p_1, p_2,  \lambda$ satisfying the assumptions of three parts \eqref{main estimate 1}, \eqref{main estimate 4} and \eqref{main estimate 5} in Proposition \ref{T_j estimate}, 
	$$\|T_{1,12}^{\lambda}(f,g)\|_p \lesssim \|f\|_{p_1}\|g\|_{p_2}$$
	holds for $f \in L^{p_1}(\R^2)$, $g \in L^{p_2}(\R^2)$.
\end{proposition}

\subsection{Proof of Proposition~\ref{prop1inj=1}}

Let $\widetilde{\varphi}(x)$ be a smooth function supported in $[1/4,4]$ which takes value $1$ on $[1/2,2]$. Then for $k \ge 2$,
\begin{align*}
T_{1,k}^{\lambda}(f,g)(x) &=\int_{\mathbb{R}^2\times \mathbb{R}^2} \psi\Big(2^k\big(1-\frac{\eta^2_1}{\eta^2_2}\big)\Big) \psi_1\Big(\frac{\xi^2_1}{\xi^2_2}\Big) \Big(1-\frac{\xi^2_1}{\xi^2_2}-\frac{\eta^2_1}{\eta^2_2}\Big)^{\lambda}_{+} \varphi(\xi_2)\widetilde{\varphi}(\xi_2)\varphi(\eta_2) \widehat{f}(\xi)\widehat{g}(\eta) e^{2\pi i x\cdot(\xi+\eta)} d\xi\, d\eta \\
& = \mathcal{C}_{k}^{\lambda}(g,h)
\end{align*}
where $\mathcal{C}_{k}^{\lambda}$ is defined in \eqref{eq:Cljoriginal} and $\widehat{h}(\xi):= \widetilde{\varphi}(\xi_2) \psi_1\big(\frac{\xi^2_1}{\xi^2_2}\big)\widehat{f}(\xi)$. The desired result is obtained by Proposition \ref{T_j estimate} with $j = k$, noting that $\|h\|_{L^p} \lesssim \|f\|_{L^p}$ holds since the term $\widetilde{\varphi}(\xi_2)\psi_1(\frac{\xi^2_1}{\xi^2_2})$ in the multiplier is harmless.




\subsection{Proof of Proposition~\ref{T1j operator}}

Note that $\operatorname{supp}(\psi_{12})\subset [\frac{1}{16},\frac{3}{4}]$, therefore invoking  identity (\ref{Stein-Weiss inequality}) with
$$
R^2=1-\frac{\eta_1^2}{\eta_2^2}\qquad \text{and} \qquad |m|^2=\frac{ \xi_1 ^{2}}{\xi_2^2},
$$
we get 
\begin{align*}
    m^{\lambda}_{1,12}(\xi,\eta)=c_{\mu,\nu}  \psi_{12}\Big(\frac{\eta^2_1}{\eta^2_2}\Big) \psi_1\Big(\frac{\xi^2_1}{\xi^2_2}\Big)\varphi(\xi_2)\varphi(\eta_2)\int^{\sqrt{\frac{15}{16}}}_0  \Big(1-\frac{\eta^2_1}{\eta^2_2}-t^2\Big)^{\mu-1}_{+} \Big(1-\frac{\xi^2_1} {t^2\xi^2_2}\Big)^{\nu}_{+} t^{2\nu+1}~dt.
\end{align*}
Hence, the operator $T_{1,12}^{\lambda}(f,g)(x)$ can be expressed as 
$$ 
T_{1,12}^{\lambda}(f,g)(x) = c_{\mu,\nu}  \int^{\sqrt{\frac{15}{16}}}_0 \widetilde{B}^{\mu}_{t}g(x){T}^{\nu}_{t}h(x) t^{2\nu+1}\,dt, 
$$
where 
\begin{align*}
    \widehat{\widetilde{B}^{\mu}_{t}g}(\eta):= \psi_{12}\Big(\frac{\eta^2_1}{\eta^2_2}\Big) \Big(1-\frac{\eta^2_1}{\eta^2_2}-t^2\Big)^{\mu-1}_{+}\varphi(\eta_2) \widehat{g}(\eta),
\end{align*}
${T}^{\nu}_{t}$ is defined in \eqref{eq:Toperator} and, analogously as in  Proposition~\ref{prop1inj=1}, we write $\widehat{h}(\xi):= \widetilde{\varphi}(\xi_2) \psi_1\big(\frac{\xi^2_1}{\xi^2_2}\big)\widehat{f}(\xi)$, where $\widetilde{\varphi}$ is defined in the proof of Proposition~\ref{prop1inj=1}.

For operator ${T}^{\nu}_{t}$, we provide the following square function estimate, which can be regarded as an extension of Proposition {\ref{square function Gnu}} in the case of $j=1$. The proof of this result follows exactly from the corresponding proof of Proposition {\ref{square function Gnu}}, and we omit it here.

\begin{proposition}{\label{Ttnu estimate 1}}
Let $b>-1/2$. Then
$$\Big\|\Big(\int^{\sqrt{\frac{15}{16}}}_{0}|{T}^{\nu}_{t}h(x)|^2~t^{2b}dt\Big)^{1/2}\Big\|_p \lesssim \|h\|_{p}$$
for all $ h \in L^{p}(\R^2)$, in the following cases:
\begin{enumerate}
\item If $p=2$, $\nu>-1/2$.
\item If $2 \le p \le \infty$, $\nu >0$.
\end{enumerate} 
In addition, for $\nu > -1/2$, we have
$$\Big\|\Big(\int^{1}_{\frac{1}{2}}|T^{\nu}_{t}h(x)|^2~t^{2b}dt\Big)^{1/2}\Big\|_4 \lesssim \|g\|_4.$$
\end{proposition}

The $L^p$ boundedness of the operator $\widetilde{B}^{\mu}_{t}$ is similar to that of $B^{\mu}_{j,t}$ operator. We prove the following. 

\begin{proposition}
\label{lem:Btilde}
The following holds
$$
\Big\|\Big(\int^{\sqrt{\frac{15}{16}}}_{0}|\widetilde{B}^{\mu}_{t}g(x)|^2~t^{2a}dt\Big)^{1/2}\Big\|_p \lesssim \|g\|_p
$$
for all $g\in L^p(\R^2)$, in the following cases: 
\begin{enumerate}
\item $p=2$, $\mu> 1/2$, $a>1/2$.
\item $2 \le p \le \infty$, $\mu>1$, $a>-1/2$.
\item $p=4$, $\mu> 1/2$, $a> 1/2$.
\end{enumerate}
In addition, for $2 \le p \le \infty$, $\mu > 0$, $k > 3$ and $a>-1/2$ we have
\begin{equation}{\label{Lemma73}}
\Big\|\Big(\int^{2^{-k/2}}_{2^{-(k+1)/2}}|\widetilde{B}^{\mu}_{t}g(x)|^2\,t^{2a}dt\Big)^{1/2}\Big\|_p \lesssim 2^{-k(\frac{a}{2}+\frac{1}{4})}\|g\|_p.      
\end{equation}   
\end{proposition}

\begin{proof}[Sketch of the proof]
The proof of Proposition \ref{lem:Btilde} largely follows the strategy used for the corresponding result in the previous sections. The only new ingredient arises in the proof of \eqref{Lemma73}.
Arguing as in Subsection \ref{Proof of LpGmuk}, it suffices to show that
$$
\widetilde{B}^{\mu}_{t}g(x) \lesssim \mathfrak{m}_1g(x)    
$$
holds uniformly for $t\in [0,\frac{1}{8}]$, where $\mu>0$. To establish this estimate, we introduce a slight modification of the proof of Proposition \ref{operatorB}. We now briefly describe this adjustment.

Let $\widetilde{K}_t$ denote the kernel corresponding to the $\widetilde{B}^{\mu}_{t}$ operator. Analogously  as in \eqref{decomposition of the kernel Kmuj,t}, we decompose
$$\widetilde{K}_t \chi_{\mathfrak{Q}_1} = \sum\limits_{k \ge 1}\widetilde{K}_t \chi_{\mathfrak{Q}_1}\chi_{B(0,2^{k}) \setminus B(0,2^{k-1})} + \widetilde{K}_t \chi_{\mathfrak{Q}_1}\chi_{B(0,1)}=: \sum\limits_{k \ge 1}\widetilde{K}_{t,k}^{\mathfrak{Q}_1} + \widetilde{K}_{t,0}^{\mathfrak{Q}_1},$$
where $\mathfrak{Q}_1=\{y \in \mathbb{R}^2: y_1,y_2>0\}$ is the first quadrant in $\mathbb{R}^2$. The term $\widetilde{K}_{t,0}^{\mathfrak{Q}_1}$ can be handled analogously to $K^{\mathfrak{Q}_1}_{j,t,0}$ in the proof of Proposition \ref{operatorB}. In order to deal with the term  $\widetilde{K}_{t,k}^{\mathfrak{Q}_1}$, we provide a further decomposition for $B(0,2^{k}) \setminus B(0,2^{k-1})$, that is, $B(0,2^{k}) \setminus B(0,2^{k-1})= \widetilde{A}_{k}^{\operatorname{off}} \cup \widetilde{A}_{k}^{\operatorname{diag}}$, where  
$$
\widetilde{A}_{k}^{\operatorname{off}}:=\big\{y \in \mathfrak{Q}_1:2^{k-1} \le |y| \le 2^{k}, y_2-y_1 \ge \tfrac{7}{8}2^{k-1}\big\} 
$$
and 
$$\widetilde{A}_{k}^{\operatorname{diag}}:=\{y \in \mathfrak{Q}_1:2^{k-1} \le |y| \le 2^{k}, y_2-y_1 \le \tfrac{7}{8}2^{k-1}\}.$$
With this new decomposition, we can proceed in the same way we did in the proof of  Proposition {\ref{operatorB}}. For the part $\widetilde{A}_{k}^{\operatorname{off}}$, as $y_1 \le 2^{k-1}$ always holds in this case, we could get the analogous estimate to {\eqref{easypart}}, and hence the strong maximal function domination result follows by proceeding as before. For the part $\widetilde{A}_{k}^{\operatorname{diag}}$, following the strategy in the proof of  Proposition {\ref{operatorB}}, the desired result is obtained by noting the fact that $y_1 \ge 2^{k-5}$ holds in this part.
\end{proof}

Combining Proposition \ref{Ttnu estimate 1} , Proposition \ref{lem:Btilde} and \eqref{L4Gnuk} and proceeding similarly as it is done in the proof of  Proposition \ref{T_j estimate} for $j\ge 2$, we complete the proof of Proposition \ref{T1j operator}.


\appendix
\section{Proof of Proposition \ref{directional square function estimate}}
\label{ap:proofds}

The proof of Proposition \ref{directional square function estimate} relies on the ideas of the proof of  \cite[Theorem 1]{Cor1982}. Nevertheless, the presence of the additional parameter $\alpha$ makes the reasoning more subtle because we need uniformity in this new parameter. Because of this, some parts in the proof sligthly differ from C\'ordoba's. For the sake of clarity, we present the sketch of the proof, giving details in the relevant places. We first need a couple of classical results.  Recall that $\mathfrak{m}^{s}(w):=(\mathfrak{m}(w^s))^{1/s}$, and $\mathfrak{m}$ is the strong maximal function. 
\begin{lemma}\cite{Cor1981, Carbery, Corbook}\label{summation lemma}
Suppose $\{P_\nu\}_{\nu \in \mathbb{Z}^2}$ is a lattice, meaning that $P_{\nu}$ are congruent rectangles with pairwise disjoint interiors and $\R^2=\cup_{\nu}P_{\nu}$. Then, there exists $\beta_2>0$ such that for every positive function $w$ and $1<s<2$,  we have
$$
\int_{\mathbb{R}^2}\sum\limits_{\nu}|P_\nu f|^2w \,dx \le \Big(\frac{C}{s-1}\Big)^{\beta_2} \int_{\mathbb{R}^2}|f|^2\mathfrak{m}^{s}(w) \, dx.
$$

\end{lemma}
\begin{remark}
Tracking the constant in \cite[Lemma 5.9]{Corbook}, one can show that $\beta_2$ can be chosen as $2$. 
\end{remark}

Given $0<\alpha< 2\pi$, define now the maximal function $M_N^{\alpha}$ as
$$
M_N^{\alpha}f(x)=\sup_{R\in \Sigma_N^{\alpha}}\frac{1}{|R|}\int_R|f(x-y)|\,dy,
$$
where $\Sigma_N^{\alpha}$ is the set of rectangles $R\in \R^2$, centered at the origin, whose longest side points in the direction $(\cos\frac{j\alpha}{N}, \sin\frac{j\alpha}{N})$, for some $0\le j\le N$. Let us remark that $M_N^{\alpha}=\sup_{0\le j\le N}\mathfrak{m}_{t_j}$, where $t_j=\tan (j\alpha/N)$.
\begin{lemma}{\label{maximal function result}}
Given $0<\alpha< 2\pi$, there exists $C>0$ independent of $\alpha$ and $N$ such that
$$
\|M_N^{\alpha}f\|_{2,\infty} \le C(\log N)^{1/2}\|f\|_2.
$$
\end{lemma}

\begin{proof}[Sketch of the proof]
The proof follows exactly the same argument in \cite[Theorem 5.3.5]{Grafakosmodern}, where the result is proved for $\alpha=2\pi$. The only difference appears in the statement of Lemma 5.3.6 in \cite{Grafakosmodern}: in our setting, $\omega_k$ is defined as $\omega_k:=2 \pi 2^k / N$ for $1 \le k < [\log_2(N)]$ and $\omega_{[\log_2(N)]}:= \alpha$.   
\end{proof}

\subsection{Proof of Proposition \ref{directional square function estimate}}

As explained before, we only need to consider the case $0 < \alpha <\frac{\pi}{4}$.

\subsubsection{Local case}

First,  consider the case that $\operatorname{supp} \widehat{f} \subset \{\xi \in \R^2,\, 1 \le \xi_1 \le 2\}$. Consider 
\begin{align}
& \Big\|\Big(\sum_{j=1}^N|\Theta_j^{\alpha}f|^2\Big)^{1/2} \Big\|_4^4 \nonumber 
= \sum\limits_{j,k}\int_{\mathbb{R}^2} |\Theta_j^{\alpha}f|^2|\Theta_k^{\alpha}f|^2 \, dx \nonumber \\
&\quad =\sum\limits_{|j-k|> N^{\frac12}}\int_{\mathbb{R}^2} |\Theta_j^{\alpha}f|^2|\Theta_k^{\alpha}f|^2 \, dx + \sum_{\nu=0}^{\frac12\log N} \sum\limits_{2^{-\nu}{N^{1/2}}\ge|j-k|>2^{-\nu-1}N^{1/2}}\int_{\mathbb{R}^2} |\Theta_j^{\alpha}f|^2|\Theta_k^{\alpha}f|^2 \, dx+\sum_{j=1}^N\|\Theta_j^{\alpha}f\|_4^4.   \label{local part}
\end{align} 
Now we analyze the three terms in \eqref{local part}.

The third term is easy to handle by vector-valued estimates, so let us focus on the other two terms.

To handle the first term in \eqref{local part}, we further decompose $\Theta_j^{\alpha}$ as follows. For $0 \le \gamma \le N^{1/2}-1$ and $1\le j\le N$, let 
$$
\Theta_{j,\gamma}^{\alpha}:=\{\xi \in \Theta_{j}^{\alpha},\, 1+\gamma N^{-1/2} \le \xi_1 \le 1+(\gamma+1) N^{-1/2}\}.
$$
It is not difficult to verify that there exists $C>0$ independent of $\alpha,N,j,k$ such that
$$\sum\limits_{\gamma_1,\gamma_2}\chi_{\Theta_{j,\gamma_1}^{\alpha}+\Theta_{k,\gamma_2}^{\alpha}} \le C$$
holds for every $1 \le j,k \le N$ satisfying $|j-k|>N^{1/2}$. Then, analogously as it was done to obtain \eqref{estimate2},
\begin{align*}
\sum\limits_{|j-k|>N^{1/2}}\int_{\mathbb{R}^2} |\Theta_j^{\alpha}f|^2|\Theta_k^{\alpha}f|^2 \, dx  \lesssim \big\|\sum\limits_{j,\gamma}|\Theta_{j,\gamma}^{\alpha}f|^2\big\|_2^2.
\end{align*}
Let  $t_k=\tan (k\alpha/N^{1/2})$. For a positive function $\omega \in L^2$, by following the same considerations for the sums as it was done in the estimate \eqref{estimate3}, using Lemma \ref{summation lemma} we have, for some $\beta_2>0$,
\begin{align*}
\int_{\mathbb{R}^2} \sum\limits_{j=1}^{N}\sum\limits_{\gamma=0}^{N^{1/2}}|\Theta_{j,\gamma}^{\alpha}f|^2\omega \,dx & \le\sum\limits_{k=0}^{N^{1/2}}\sum\limits_{\gamma=0}^{N^{1/2}}\sum\limits_{j=kN^{1/2}}^{(k+1)N^{1/2}}\int_{\mathbb{R}^2} |\Theta_{j,\gamma}^{\alpha}P_{k,\gamma}^{\alpha}f|^2 \omega \,dx \\
&\le \Big(\frac{C}{s-1}\Big)^{\beta_2}\sum\limits_{k=0}^{N^{1/2}}\sum\limits_{\gamma=0}^{N^{1/2}}\int_{\mathbb{R}^2} |P_{k,\gamma}^{\alpha}f|^2 \mathfrak{m}_{t_k}^s(\omega) \,dx \\
&\le \Big(\frac{C}{s-1}\Big)^{\beta_2}\sum\limits_{k=0}^{N^{1/2}}\sum\limits_{\gamma=0}^{N^{1/2}}\int_{\mathbb{R}^2} |P_{k,\gamma}^{\alpha}f|^2 \sup_{0\le k\le N^{1/2}}\mathfrak{m}_{t_k}^s(\omega) \,dx \\
& =\Big(\frac{C}{s-1}\Big)^{\beta_2}\sum\limits_{k=0}^{N^{1/2}}\sum\limits_{\gamma=0}^{N^{1/2}}\int_{\mathbb{R}^2} |P_{k,\gamma}^{\alpha}f|^2 \operatorname{M}_{N^{1/2}}^{\alpha,s}(\omega) \,dx, 
\end{align*}
where $\{P_{k,\gamma}^{\alpha}\}_{k,\gamma}$ is a family of rectangles with dimension $10\alpha N^{-1/2} \times N^{-1/2}$ and $\operatorname{M}_{N^{1/2}}^{\alpha,s}(\omega)=(\operatorname{M}_{N^{1/2}}^{\alpha}(\omega^s))^{1/s}$. Applying again Lemma~\ref{summation lemma} we get that 
$$\sum\limits_{k=0}^{N^{1/2}}\sum\limits_{\gamma=0}^{N^{1/2}}\int_{\mathbb{R}^2} |P_{k,\gamma}^{\alpha}f|^2 \operatorname{M}_{N^{1/2}}^{\alpha,s}(\omega) \,dx \le \Big(\frac{C}{s-1}\Big)^{\beta_2}\int_{\mathbb{R}^2} |f|^2 \mathfrak{m}^s \circ \operatorname{M}_{N^{1/2}}^{\alpha,s}(\omega) \,dx.
$$
 As a result, there exists a positive function $\omega$ with $\|\omega\|_2=1$ such that
\begin{align*}
 \sum\limits_{|j-k|>N^{1/2}}\int_{\mathbb{R}^2} |\Theta_j^{\alpha}f|^2|\Theta_k^{\alpha}f|^2 \, dx & \lesssim \|\sum\limits_{j,\gamma}|\Theta_{j,\gamma}^{\alpha}f|^2\|_2^2  \\
& =  \Big(\int_{\mathbb{R}^2}\sum\limits_{j,\gamma}|\Theta_{j,\gamma}^{\alpha}f|^2 \omega \,dx \Big)^2 \\
& \lesssim \frac{1}{(s-1)^{4\beta_2}}\Big(\int_{\mathbb{R}^2} |f|^2 \mathfrak{m}^s \circ \operatorname{M}_{N^{1/2}}^{\alpha,s}(\omega) \,dx\Big)^2 \\
&\le \frac{1}{(s-1)^{4\beta_2}} \|f\|_4^4 \|\mathfrak{m}^s  \circ \operatorname{M}_{N^{1/2}}^{\alpha,s}(\omega)\|_2^2. 
\end{align*}
Taking $s=1+1/\log N$ and interpolating between Lemma {\ref{maximal function result}} (with $N^{1/2}$) and the trivial bound
$$\|\operatorname{M}_{N^{1/2}}^{\alpha}f\|_{3/2} \le CN^{1/3}\|f\|_{3/2},$$
we know that
\begin{align*}
\|\operatorname{M}_{N^{1/2}}^{\alpha,s}(\omega)\|_2 = \|\operatorname{M}_{N^{1/2}}^{\alpha}(\omega^{s})\|_{2/s}^{1/s} &\le C N^{1/\log N} (\log N)^{1/2}\|\omega^s\|_{2/s}^{1/s}\\
&\le C (\log N)^{1/2}\|\omega^s\|_{2/s}^{1/s} = C (\log N)^{1/2} \|\omega\|_{2} = C(\log N)^{1/2},
\end{align*}
where $C$ is a constant independent of $N$. Thus, 
$$  \sum\limits_{|j-k|>N^{1/2}}\int_{\mathbb{R}^2} |\Theta_j^{\alpha}f|^2|\Theta_k^{\alpha}f|^2 \, dx  \le C(\log N)^{4\beta_2+1} \|f\|_4^4$$
and we are done with the first term in \eqref{local part}.


For the second term in (\ref{local part}) we proceed analogously as in \cite{Cor1982}. We decompose $\Theta_j^{\alpha}$ as follows. For $0 \le \gamma \le 2^{-\nu}N^{1/2}-1$, let 
$$\Theta_{j,\gamma}^{\alpha}:=\{\xi \in \Theta_{j}^{\alpha}, 1+\gamma 2^{\nu}N^{-1/2} \le \xi_1 \le 1+(\gamma+1) 2^{\nu}N^{-1/2}\}.
$$
It is not difficult to verify that there exists $C>0$ independent of $\alpha,N,\nu,j,k$ such that
$$\sum\limits_{\gamma_1,\gamma_2}\chi_{\Theta_{j,\gamma_1}^{\alpha}+\Theta_{k,\gamma_2}^{\alpha}} \le C
$$
holds for every $0 \le j,k \le N$ satisfying $2^{-\nu-1}N^{1/2}< |j-k| < 2^{-\nu}N^{1/2}$,  and we can proceed by combining the strategy for the first term in \eqref{local part} and vector-valued estimates.

Combining estimates of the three parts in \eqref{local part}, we get that 
\begin{equation}{\label{resultsinA12}}
\Big\|\Big(\sum_{j=1}^N|\Theta_j^{\alpha}f|^2\Big)^{1/2}\Big\|_4 \le (\log N)^{\beta_2+1/4}\|f\|_4.    
\end{equation}
By homogeneity, estimate \eqref{resultsinA12} also holds for $\operatorname{supp} \widehat{f} \subset \{\xi \in \R^2, 2^n \le \xi_1 \le 2^{n+1}\}=:\Delta_n$ with $n \in \Z$.

\subsubsection{Global case}

We upgrade the estimate to general functions (not only supported in $2^n \le \xi_1\le 2^{n+1}$ for some $n \in \Z$). Let $P_{j,n}^{\alpha}:=\Theta_j^{\alpha} \cap \Delta_n$, we first decompose the square function into $\log N$ parts,
\begin{align*}
\Big(\sum_{j=1}^N|\Theta_j^{\alpha}f|^2\Big)^{1/2} = \Big(\sum_{j=1}^N|\sum\limits_{n \in \mathbb{Z}}P_{j,n}^{\alpha}f|^2\Big)^{1/2} & \le \Big(\sum_{j=1}^N|\sum\limits_{\ell=0}^{\log N-1}\sum\limits_{n \in \mathbb{Z}}P_{j,n\log N+\ell}^{\alpha}f|^2\Big)^{1/2} \\
& \le \sum\limits_{\ell=0}^{\log N-1}\Big(\sum_{j=1}^N|\sum\limits_{n \in \mathbb{Z}}P_{j,n\log N+\ell}^{\alpha}f|^2\Big)^{1/2}.  
\end{align*}
By symmetry, we only need to consider $\ell=0$. Using the $\ell^2$-valued dyadic Littlewood--Paley theorem,
\begin{align}
&
\Big\|\Big(\sum_{j=1}^N|\sum\limits_{n \in \mathbb{Z}}P_{j,n\log N}^{\alpha}f|^2\Big)^{1/2}\Big\|_4^4 \nonumber\simeq  \Big\|\Big(\sum_{j=1}^N\sum\limits_{n \in \mathbb{Z}}|P_{j,n\log N}^{\alpha}f|^2\Big)^{1/2}\Big\|_4^4 \nonumber \\
&\quad=  \sum\limits_{n \in \mathbb{Z}}\int_{\mathbb{R}^2} \Big(\sum\limits_{j=1}^{N}|P_{j,n}^{\alpha}f|^2\Big)^2 \, dx + 2  \sum\limits_{n_1 > n_2}\sum\limits_{j,k}\int_{\mathbb{R}^2}|P_{j,n_1 \log N}^{\alpha}f|^2|P_{k,n_2 \log N}^{\alpha}f|^2 \, dx. \label{middle step}
\end{align}
For the first part in \eqref{middle step}, since $\operatorname{supp} \widehat{\Delta_n f} \subset \Delta_n$, we can apply estimate \eqref{resultsinA12} and obtain
\begin{align}
\sum\limits_{n \in \mathbb{Z}}\int_{\mathbb{R}^2} \Big(\sum\limits_{j=1}^{N}|P_{j,n}^{\alpha}f|^2\Big)^2 \, dx \le &  
(\log N)^{4\beta_2+1} \sum\limits_{n \in \mathbb{Z}}\int_{\mathbb{R}^2}|\Delta_nf|^4 \, dx {\label{estimate4}} \\
\le & (\log N)^{4\beta_2+1} \int_{\mathbb{R}^2}\Big(\sum\limits_{n \in \mathbb{Z}}|\Delta_nf|^2 \Big)^2  \, dx\nonumber \\
\le & (\log N)^{4\beta_2+1} \|f\|_4^4.\nonumber
\end{align}
For the second part in (\ref{middle step}), we make a further decomposition on $P_{j,n}^{\alpha}$. For $0 \le \gamma \le N-1$, define
$$P_{j,n,\gamma}^{\alpha}:=\{\xi \in P_{j,n}^{\alpha},\, 2^n + \gamma N^{-1} < \xi_1 < 2^n +(\gamma+1) N^{-1}\}.$$
Once again, it can verified that there exists $C>0$ independent of $\alpha, n_1, n_2,j,k$ such that 
$$\sum\limits_{\gamma}\chi_{P_{j,n_1 \log N,\gamma}^{\alpha}+{P_{k,n_2 \log N}^{\alpha}}} \le C$$
holds for every $n_1, n_2 \in \mathbb{Z}$, $1 \le j,k \le N$ and $n_1 > n_2$. Then, similarly as we did to obtain \eqref{estimate2}, we get
\begin{equation}
 \sum\limits_{n_1 > n_2}\sum\limits_{j,k}\int_{\mathbb{R}^2}|P_{j,n_1 \log N}^{\alpha}f|^2|P_{k,n_2 \log N}^{\alpha}f|^2 \, dx  \lesssim \Big\|\Big(\sum\limits_{n,j,\gamma}|P_{j,n,\gamma}^{\alpha}f|^2\Big)^{1/2}\Big\|_4^2\Big\|\Big(\sum\limits_{n_2,k}|P_{k,n_2 \log N}^{\alpha}f|^2\Big)^{1/2}\Big\|_4^2. \label{estimate6}
\end{equation}
Choosing a suitable positive function $\omega$ with $\|\omega\|_2=1$, we have 
\begin{align}
\notag \Big\|\Big(\sum\limits_{n,j,\gamma}|P_{j,n, \gamma}^{\alpha}f|^2\Big)^{1/2}\Big\|_4^2  = \int_{\mathbb{R}^2} \sum\limits_{n}\sum\limits_{j,\gamma}|P_{j,n, \gamma}^{\alpha}f|^2 \omega \, dx
& \le \int_{\mathbb{R}^2} \sum\limits_{n}|\Delta_{n}f|^2 \mathfrak{m}^{3/2}(\omega) \, dx \nonumber \\
& \le \Big\|\Big(\sum\limits_{n}|\Delta_{n}f|^2\Big)^{1/2}\Big\|_4^2\|\mathfrak{m}^{3/2}(\omega)\|_2 \nonumber \\
& \le \|f\|_4^2, \label{estimate5}
\end{align}
where we used Lemma \ref{summation lemma} with $s=\frac{3}{2}$ and the classical Littlewood--Paley result in the above estimate.

Combining estimates \eqref{middle step}, \eqref{estimate4}, \eqref{estimate6} and \eqref{estimate5}, we have
$$ 
\Big\|\Big(\sum_{j=1}^N\sum\limits_{n \in \mathbb{Z}}|P_{j,n\log N}^{\alpha}f|^2\Big)^{1/2}\Big\|_4^4 \lesssim (\log N)^{4\beta_2 +1}\|f\|_4^4 +\|f\|_4^2\Big\|\Big(\sum_{j=1}^N\sum\limits_{n \in \mathbb{Z}}|P_{j,n\log N}^{\alpha}f|^2\Big)^{1/2}\Big\|_4^2.$$
From this we know that there exists $\beta_1 > 0$ such that 
$$
\Big\|\Big(\sum_{j=1}^N|\Theta_j^{\alpha}f|^2\Big)^{1/2}\Big\|_4 \le (\log N)^{\beta_1}\|f\|_4,
$$
and the proof is finished.
\section*{Acknowledgements} 
The authors are thankful to Antonio C\'ordoba for sending a copy of his manuscript \cite{Cor1981}. Luz Roncal is supported 
by the Spanish Agencia Estatal de Investigaci\'{o}n
through BCAM Severo Ochoa accreditation CEX2021-001142-S/MCIN/AEI/10.13039/501100011033, by 
CNS2023-143893 and by PID2023-146646NB-I00 funded by MICIU/AEI/10.13039/501100011033 and FEDER/EU and by ESF+, and by the Basque Government through the BERC 2022--2025 program. Saurabh Shrivastava is supported by Anusandhan National Research Foundation, India under the project  ANRF/ARG/2025/000940/MS. 
Kalachand Shuin gratefully acknowledges the support of Inspire Faculty Award of Department of Science
and Technology, Govt. of India, Registration no. IFA23-MA191. Linfei Zheng is supported by the National Key R\&D Program of China (Grant No. 2021YFA1002500) and China Scholarship Council.

\bibliography{biblio}
\end{document}